\documentclass[review,hidelinks,onefignum,onetabnum]{siamart251216}

\usepackage[T1]{fontenc}
\usepackage[utf8]{inputenc}
\usepackage{amsfonts,amssymb}
\usepackage{mathrsfs,dsfont,bbm}
\usepackage{mathtools}
\usepackage{multibib}
\usepackage{graphicx}
\hypersetup{colorlinks=true, linkcolor=blue, urlcolor=blue, citecolor=blue}

\newcommand{\be}[1]{\begin{equation}\label{#1}}
\newcommand{\ee}{\end{equation}}
\renewcommand{\(}{\left(}
\renewcommand{\)}{\right)}
\newcommand{\bangle}[1]{\langle #1\rangle}
\newcommand{\wangle}[1]{\lfloor #1 \rceil}
\newcommand{\msc}[1]{\href{https://zbmath.org/classification/?q=cc:#1}{#1}}
\newcommand{\step}[2]{\par\medskip\noindent\emph{\textbf{Step #1.} #2}}
\newcounter{Step}0\newcommand{\Step}[1]{\addtocounter{Step}1\par\medskip\noindent\emph{\textbf{Step \theStep.} #1}}

\newcommand{\R}{\mathbb R}
\newcommand{\Rd}{{\R^d}}
\newcommand{\RRd}{{\Rd\times\Rd}}
\newcommand{\one}{\mathbbm 1}
\newcommand{\dd}{\,\mathrm d}
\newcommand{\cL}{\mathcal L}
\newcommand{\cA}{\mathcal A}
\newcommand{\cB}{\mathcal B}
\newcommand{\cC}{\mathcal C}
\newcommand{\cH}{\mathcal H}

\newcommand{\sfL}{\mathsf{L}}
\newcommand{\sfT}{\mathsf{T}}
\newcommand{\sfH}{\mathsf{H}}
\newcommand{\sfA}{\mathsf{A}}
\newcommand{\sfD}{\mathsf{D}}
\newcommand{\sfPi}{\mathsf{\Pi}}
\newcommand{\rma}{\mathrm a}
\newcommand{\rmL}{\mathrm L}
\newcommand{\wx}{\lfloor x \rceil}
\newcommand{\wv}{\lfloor v \rceil}
\def\dx{\,\mathrm dx}

\def\dv{\,\mathrm dv}

\newcommand{\nrm}[2]{\left\|{#1}\right\|_{#2}}
\newcommand{\irdx}[1]{\int_\Rd{#1}\dx}
\newcommand{\irdv}[1]{\int_\Rd{#1}\dv}
\newcommand{\irdxv}[1]{\iint_{\RRd}#1\dx\dv}
\newcommand{\irdxvfstar}[1]{\iint_{\RRd}#1\dd\widetilde\mu_\star}
\newcommand{\irdmustar}[1]{\iint_{\RRd}#1\dd\mu_\star}
\newcommand{\nrmxv}[2]{\|#1\|_{\rmL^{#2}(\dd x\dv)}}

\newcommand{\tildemeas}[1]{\dd\widetilde\sigma_{#1,\varepsilon}}
\DeclarePairedDelimiter\norm{\lVert}{\rVert}
\DeclarePairedDelimiter\abs{\lvert}{\rvert}

\definecolor{darkgreen}{rgb}{.0,.4,.2}

\nolinenumbers
\makeatletter\renewcommand\tableofcontents{\section*{\contentsname}\@starttoc{toc}}\makeatother
\DeclareUnicodeCharacter{015F}{\c s}
\begin{document}
\ifpdf\hypersetup{pdftitle={Weighted \texorpdfstring{$\rmL^2$}{L2} bounds for kinetic Fokker-Planck equations and hypocoercivity}, pdfauthor={E.~Bouin, J.~Dolbeault, and L.~Ziviani}}\fi

\title{Weighted \texorpdfstring{$\rmL^2$}{L2} bounds for kinetic Fokker-Planck equations and hypocoercivity}
\headers{Weighted \texorpdfstring{$\rmL^2$}{L2} bounds and hypocoercivity}{E.~Bouin, J.~Dolbeault, and L.~Ziviani}

\author{
Emeric Bouin \thanks{E.~B (\email{bouin@ceremade.dauphine.fr}, \url{https://www.ceremade.dauphine.fr/\string~bouin/})}
\and Jean Dolbeault \thanks{J.~D.(\email{dolbeaul@ceremade.dauphine.fr}, \url{https://www.ceremade.dauphine.fr/\string~dolbeaul/})}
\and Luca Ziviani \thanks{L.~Z. (\email{luca.ziviani@dauphine.eu}, \url{https://luca-ziviani.github.io})\\
CEREMADE (CNRS UMR n$^\circ$ 7534), Universit\'e Paris-Dauphine, PSL University,\\ Place de Lattre de Tassigny, 75775 Paris 16, France}}

\maketitle
\thispagestyle{empty}

\begin{abstract} We study the long-time behaviour of solutions to kinetic Fokker-Planck equations with power law confinement potentials and local equilibria with fat tails. Without relying on \emph{a priori} moment bounds or perturbative regimes, we establish global estimates on weighted norms (main result) using two methods: a coupling of the weighted norms with entropy dissipation (DMS method) and a Foster-Lyapunov condition approach. As a consequence, we establish an entropy - entropy production inequality and the convergence rates to equilibrium in terms of the exponents associated to the growth of the spatial confinement and the tail decay of local equilibria. For sake of simplicity, we assume that stationary states are factorized.
\end{abstract}



\section{Introduction and main results}\label{Sec:intro}

Over the last 15 years, large time properties of non-homogeneous linear kinetic equations have been intensively studied. Apart perturbative methods, hypocoercivity results are nowadays based on several methods involving various functional frameworks ranging from $\mathrm H^1$ estimates using commutator methods inspired by hypoelliptic results or differential systems of estimates~\cite{MR2214953,MR2562709,MR3884261,MR4226238}, to $\rmL^2$ results~\cite{MR2215889,MR2576899,MR3324910,MR3557714,MR4063917,MR4069622,MR4278431} and $\mathrm H^{-1}$ estimates~\cite{MR4776290,MR4632836,MR4557687,MR4927106}. Some insight on the underlying structure can be gained from micro-macro decompositions and diffusion limits, with the guiding idea that:
\begin{enumerate}
\item[(1)] the relaxation to local equilibria occurs on time scales controlled by the collision operator acting on velocities, which can for instance take the form of diffusion or scattering operators,
\item[(2)] the equilibration in positions is achieved through the transport operator, via commutators or using a \emph{twisted} norm built on the diffusion limit.
\end{enumerate}
To fix ideas, we shall consider the case of the whole $d$-dimensional Euclidean space in positions and velocity with a confinement potential and a collision operator such that there is a stationary solution of finite mass. A classification has emerged in a series of papers~\cite{MR4113786,MR4200970,MR4063911,MR4389085,MR4769957,MR5032985}, which relies on two functional inequalities associated respectively with (1) and (2). These two inequalities determine the rates of convergence and are deeply related respectively with the decay of the local equilibria for large values of the velocity variable and the growth rate of the confinement potential for large values of the position variable.

At this point appears a first difficulty. While in classical kinetic models involving Maxwellian equilibria there is usually a \emph{factorisation property} (stationary solutions are in the intersection of the kernels of the transport and collision operators), this is not anymore true in the general case. There is another issue due to the lack of coercivity of the collision operator, which can however be compensated by \emph{moment estimates}. Here our goal is to give a \emph{complete classification} of the convergence rates to stationary states with the factorization property, which covers sub-exponential behaviours at infinity in velocities or positions, \emph{i.e.}, when the logarithm of the stationary states has (in absolute value) a sub-quadratic growth. We refer to~\cite{bouin2025convergencenonexplicitsteadystate,bouin2026convergencenonexplicitsteadystate} for recent progresses when the factorization property does not hold. Concerning moments and weighted $\rmL^2$ norms in the context of $\rmL^2$-hypocoercivity, we are aware of~\cite{MR4389085,MR4200970}, but only in the case of the free transport operator. The main novelty of this paper is the control of the moments in positions and velocities and corresponding weighted $\rmL^2$ norms, in presence of a confining potential.

\medskip We consider the \emph{kinetic Fokker-Planck} (or Vlasov-Fokker-Planck) equation
\be{kFP}
\partial_t f=\cL f\,,\quad(t,x,v)\in\R^+\times\RRd
\ee
and denote by $S_\cL$ the corresponding semi-group, such that
\[
f(t,\cdot,\cdot)=S_\cL(t)f_0
\]
the solution with nonnegative initial datum $f_0$. The \emph{transport operator}~$\sfT$ and \emph{diffusion operator} $\sfL$ are such that
\[
\cL :=\sfL - \sfT\,,\quad\sfT f=\nabla_v \psi\cdot\nabla_x f-\nabla_x \phi\cdot \nabla_v f\quad\mbox{and}\quad \sfL f=\nabla_v \cdot \big(\nabla_v f+f\,\nabla_v \psi\big)
\]
where, with the notation $\wangle{\cdot}:=\sqrt{1+|\cdot|^2}$,
\[
\phi(x)=\frac{\wx^\alpha}{\alpha}\quad\mbox{and}\quad\psi(v)=\frac{\wv^\beta}\beta
\]
for positive reals numbers $\alpha$ and $\beta$, so that $\sfT E=0$ and $\sfL\,e^{-E}=0$ if $E$ is the \emph{Hamiltonian energy} given by
\[
E(x,v):=\phi(x)+\psi(v)\,.
\]
Notice that $\sfT$ corresponds to the Poisson brackets of $E$. The function
\[
f_\star(x,v):=Z^{-1}\,e^{-E(x,v)}
\]
normalized by the condition $\irdxv{f_\star(x,v)}=1$ is such that $\sfL\,f_\star=\sfT\,f_\star=0$. It is therefore stationary and satisfies the factorization property, which in our case means that $f_\star$ is the product of a local equilibrium $e^{-\psi}$ by the spatial spatial density of the equilibrium $e^{-\phi}$, up to normalization constants. On the Hilbert space $\rmL^2(\mathrm d\mu_\star)$, the operators~$\sfT$ and $\sfL$ are respectively skew-symmetric and self-adjoint with respect to $\bangle{\cdot,\cdot}_{\rmL^2(\mathrm d\mu_\star)}$, where we adopt the notation
\[
\mathrm d\mu_\star:=f_\star^{-1}\dx\dv\quad\mbox{and}\quad\mathrm d\widetilde\mu_\star:=f_\star\dx\dv\,.
\]
It is a classical hypocoercivity result (see~\cite{MR3324910}, and Section~\ref{Sec:loss} for a sketch of a proof) that
\be{Ineq:Exp}
\norm{f-f_\star}_{\rmL^2(\mathrm d\mu_\star)}^2\le\mathscr C\,e^{-\lambda\,t}\,\norm{f_0-f_\star}_{\rmL^2(\mathrm d\mu_\star)}^2\quad\forall\,t\ge0
\ee
for some positive constants $\mathscr C$ and $\lambda$ if $\alpha\ge1$ and $\beta\ge1$. Otherwise, we have the following result.
\begin{proposition}{\rm~\cite{MR4769957}}\label{prop:rates} Let $f\!\in\!\mathrm C^0\big(\R^+;\rmL^1(\dd x\dv)\cap\rmL^2(\mathrm d\mu_\star)\big)$ be a solution to~\eqref{kFP} with nonnegative initial datum $f_0$ such that
\[
\mathscr K[f]:=\sup_{t\in\R^+}\iint_\RRd |f(t,x,v)|^2\,\big(\wx^{\alpha\,k}+\wv^{\beta\,\ell}\big)\,\mathrm d\mu_\star
\]
is finite. There is a constant $\mathscr C>0$ which depends only on $k$, $\ell$ and $\mathscr K[f]$ such that
\[
\norm{f-f_\star}_{\rmL^2(\mathrm d\mu_\star)}^2\le\mathscr C\,(1+t)^{-\,\zeta/2}\,\norm{f_0-f_\star}_{\rmL^2(\mathrm d\mu_\star)}^2\quad\forall\,t\ge0
\]
in each of the following cases:
\begin{enumerate}
\item[\rm(i)] $k=0$, $\ell>0$ and $\zeta=\ell/(1-\beta)$ if $\alpha\ge1$ and $\beta\in(0,1)$,
\item[\rm(ii)] $k>0$, $\ell=0$ and $\zeta=k/(1-\alpha)$ if $\alpha\in(0,1)$ and $\beta\ge1$,
\item[\rm(iii)] $k>0$, $\ell>0$ and $\zeta=\min\{k/(1-\alpha),\,\ell/(1-\beta)\}$ if $(\alpha,\beta)\in(0,1)^2$.
\end{enumerate}
\end{proposition}
The result of Proposition~\ref{prop:rates} is conditional because $\mathscr K[f]<+\infty$ is not \emph{a priori} granted. See Sections~\ref{Sec:loss} and~\ref{Sec:Energy} for a sketch of the proof of Proposition~\ref{prop:rates}. In~\cite[Theorem~2]{MR4769957}, it is proved that $\mathscr K$ is finite under the condition that $f_0\lesssim f_\star$, which is then propagated to $f(t,\cdot,\cdot)$ for any $t>0$ by the Maximum Principle. Our task is to get rid of such a restrictive condition. Here we use the notation $\mathsf a\lesssim\mathsf b$ if there exist a constant $c>0$ such that $\mathsf a\le c\,\mathsf b$ and we shall write $\mathsf a\asymp\mathsf b$ if $\mathsf a\lesssim\mathsf b$ and $\mathsf b\lesssim\mathsf a$ hold simultaneously.
\begin{theorem}\label{Thm:moments} Let $f\in\mathrm C^0\big(\R^+;\rmL^1(\dx\,\dv)\cap\rmL^2(\mathrm d\mu_\star)\big)$ be a solution to~\eqref{kFP} with nonnegative initial datum $f_0\in\rmL^1(\dx\,\dv)\cap\rmL^2(m^k\,\mathrm d\mu_\star)$ for some continuous weight $m\ge1$ on $\RRd$ and $k>0$. Then there is a positive constant $\mathscr C>0$ which does not depend on $f_0$ such that
\[
\sup_{t\in\R^+}\iint_\RRd|f(t,x,v)|^2\,m^k(x,v)\,\mathrm d\mu_\star\le\mathscr C\iint_\RRd|f_0|^2\,m^k\,\mathrm d\mu_\star
\]
in each of the following cases:
\begin{enumerate}
\item[\rm(a)] $m(x,v)=E(x,v)$ if $\min\{\alpha,\beta\}>2/3$, or $\alpha\ge 1$ and $\beta>0$,
\item[\rm(b)] $m(x,v)=\wx^{\max\{\alpha,2\}}+\wv^{\max\{\beta,2\}}$ if $\alpha>0$ and $\beta>1$,
\item[\rm(c)] $m(x,v)=E(x,v)$ if $\alpha\in(0,1)$ and $\beta\in(0,1]$.
\end{enumerate}
\end{theorem}
In each of the three cases $m$ is \emph{coercive} in the sense that $\lim_{|(x,v)|\to+\infty}m(x,v)=+\infty$. With the results of Theorem~\ref{Thm:moments}, we can revisit Proposition~\ref{prop:rates} with $\ell=k$ and get rid of the formal condition that $\mathscr K[f]$ is finite.
\begin{corollary}\label{Cor:rates} Let $f\in\mathrm C^0\big(\R^+;\rmL^1(\dx\,\dv)\cap\rmL^2(\mathrm d\mu_\star)\big)$ be a solution to~\eqref{kFP} with nonnegative initial datum $f_0\in\rmL^1(\dx\,\dv)\cap\rmL^2(m^k\,\mathrm d\mu_\star)$. There is a constant $\mathscr C_\star>0$ which depends only on $k$, $\mathscr C$ and $\norm{f_0}_{\rmL^2(m^k \mathrm d\mu_\star)}$ such that
\[
\norm{f-f_\star}_{\rmL^2(\mathrm d\mu_\star)}^2\le\mathscr C_\star\,(1+t)^{-\,\zeta/2}\,\norm{f_0-f_\star}_{\rmL^2( \mathrm d\mu_\star)}^2\quad\forall\,t\ge0
\]
with $m$ as in Theorem~\ref{Thm:moments} and $\zeta$ such that:
\begin{enumerate}
\item[\rm(a)] $\zeta=k/\max\{1-\alpha,1-\beta\}$ if $(\alpha,\beta)\in(2/3,1)^2$, \\and $\zeta=k/(1-\beta)$ if $\alpha\ge 1$ and $\beta\in(0,1)$,
\item[\rm(b)] $\zeta=k/(1-\alpha)$ if $\alpha\in(0,1)$ and $\beta>1$,
\item[\rm(c)] $\zeta=k/\max\{1-\alpha,1-\beta\}$ if $(\alpha,\beta)\in(0,1)\times(0,1]$.
\end{enumerate}
\end{corollary}

Form the physics point of view, $\beta=2$ and $\beta=1$ make sense because of, respectively, classical and relativistic Hamiltonian mechanics, but other exponents $\beta$ can be considered, for instance in solid state physics. Confinement, which is more related to macroscopic modelling, can take various behaviours, \emph{i.e.}, any $\alpha>0$. However, our motivation here is mostly the completion of the classification of $\rmL^2$ functional inequalities, with a focus on functional inequalities and on the last pending issue: moment estimates and weighted $\rmL^2$ norms. Mathematically, there are many related issues which are not developed here: non-kinetic Fokker-Planck diffusions with fat tail equilibria (see~\cite{MR4769957} for a review), weak Poincar\'e estimates~\cite{MR4265692}, semi-groups theory and splitting methods~\cite{MR3488535}, connections with Doeblin’s and Harris' theorems~\cite{MR4063917}, \emph{etc}. Like in~\cite{MR4769957}, several related results rely on uniform \emph{a priori} assumptions, see~\cite{MR4557687,MR4927106,brigati2025explicitconvergenceratesunderdamped}.

Depending on the values of $\alpha$ and $\beta$, we adopt two different strategies (see Figure~\ref{fig:alpha_beta}): (1) If $\min\{\alpha,\beta\}>2/3$ or $(\alpha,\beta)\in[1,+\infty)\times(0,1)$, we adapt the $\rmL^2$ hypocoercivity method of~\cite{MR3324910} and establish a system of differential inequalities (coupling) which combines moments or weighted $\rmL^2$ norms, and an entropy: see Section~\ref{Sec:Coupling}. Notice that~\cite{MR3324910} covers the case $\min\{\alpha,\beta\}\ge1$ without moment estimates. (2) The case $(\alpha,\beta)\in(0,1)\times(0,+\infty)$ is covered in Section~\ref{Sec:Lyapunov} using functions satisfying the \emph{Lyapunov condition} and a Duhamel formula to prove a priori estimates on the weighted $\rmL^2$ norms, which is then sufficient to implement the standard $\rmL^2$ hypocoercivity method. The Lyapunov condition, also called \emph{Foster-Lyapunov criterion}, is a sufficient condition that goes back to~\cite{MR56232,MR84889} and has been succesfully implemented in kinetic equations, see for instance~\cite{MR2381160,MR4063917}. The Duhamel formula (see Section~\ref{Sec:Moments}) is taken from~\cite{MR4200970,MR4389085}. The main difficulty is to check the \emph{Lyapunov condition}: here our approach is mostly inspired by~\cite{MR4063917}, which itself relies on~\cite{MR84889,MR2509253,Hairer_2011}.

\setlength\unitlength{1cm}
\begin{figure}[ht]
\begin{center}
\begin{picture}(12,5)
\linethickness{1pt}
\put(0,0){\includegraphics[width=5cm]{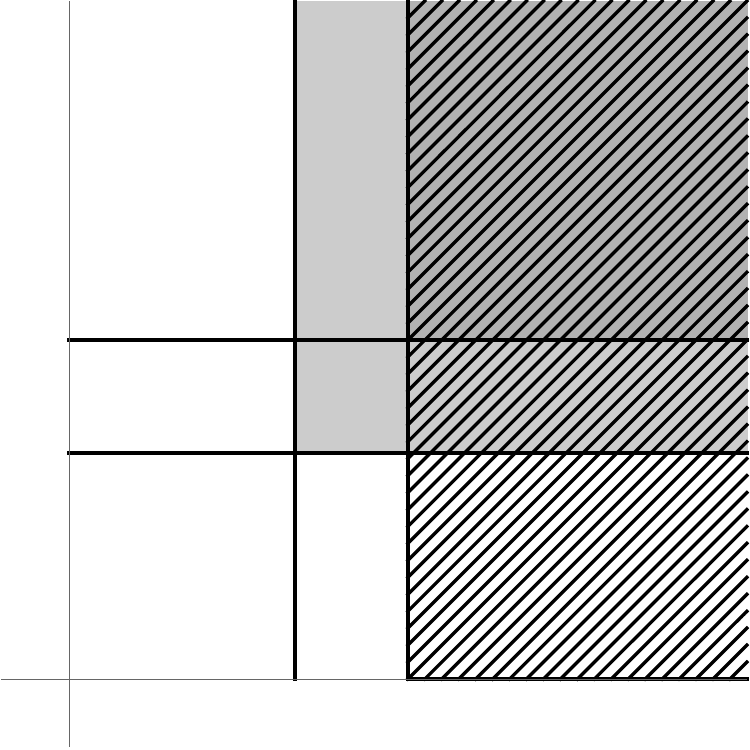}\hspace*{2cm}\includegraphics[width=5cm]{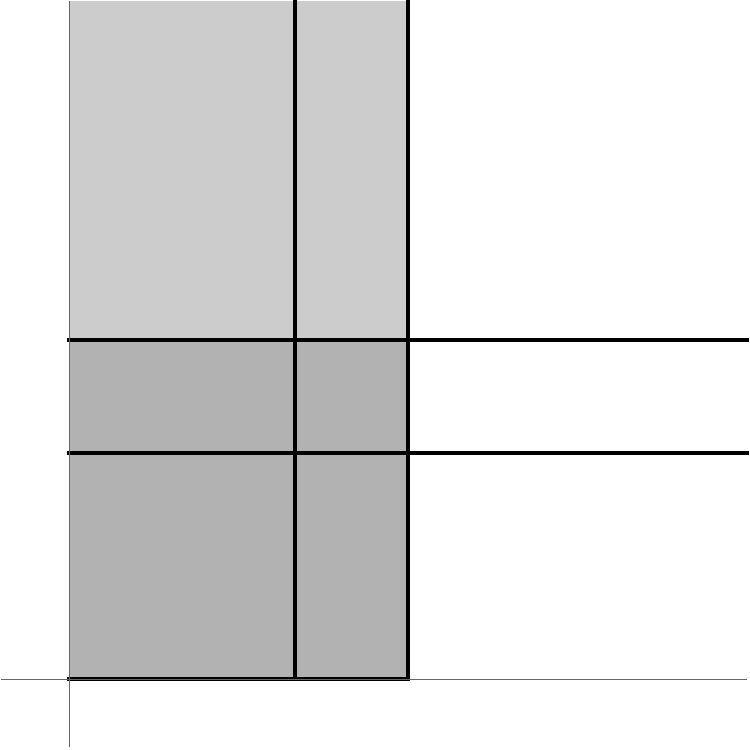}}
\put(0,.48){\vector(1,0){5.5}}
\put(7,.48){\vector(1,0){5.5}}
\put(0.46,0){\vector(0,1){5}}
\put(7.46,0){\vector(0,1){5.25}}
\put(0.2,0.15){$0$}
\put(1.8,0.15){$\frac23$}
\put(2.55,0.15){$1$}
\put(5.35,0.15){$\alpha$}
\put(7.25,0.15){$0$}
\put(8.8,0.15){$\frac23$}
\put(9.55,0.15){$1$}
\put(12.35,0.15){$\alpha$}
\put(0.2,1.85){$\frac23$}
\put(0.2,2.55){$1$}
\put(0.1,4.9){$\beta$}
\put(7.2,1.85){$\frac23$}
\put(7.2,2.55){$1$}
\put(7.1,4.9){$\beta$}
\end{picture}
\vspace*{-0.2cm}
\end{center}
\label{fig:alpha_beta}
\caption{\small \emph{Left: coupling} of energy moments and $\rmL^2$ entropy estimates. Dark grey: exponential decay, no moment needed (Section~\ref{Sec:loss}). Light grey: coupling with standard $\rmL^2$ entropy estimates (Section~\ref{Sec:Energy}). Hatched area: coupling with modified $\rmL^2$ entropy estimates and interpolation (Section~\ref{Sec:modifiedDMS}).
\emph{Right:} Moments based on a Lyapunov condition and Duhamel estimates. Light grey (Section~\ref{Sec:Lyapunov1}) and dark grey (Section~\ref{Sec:Lyapunov2}).}
\vspace*{-0.3cm}
\end{figure}
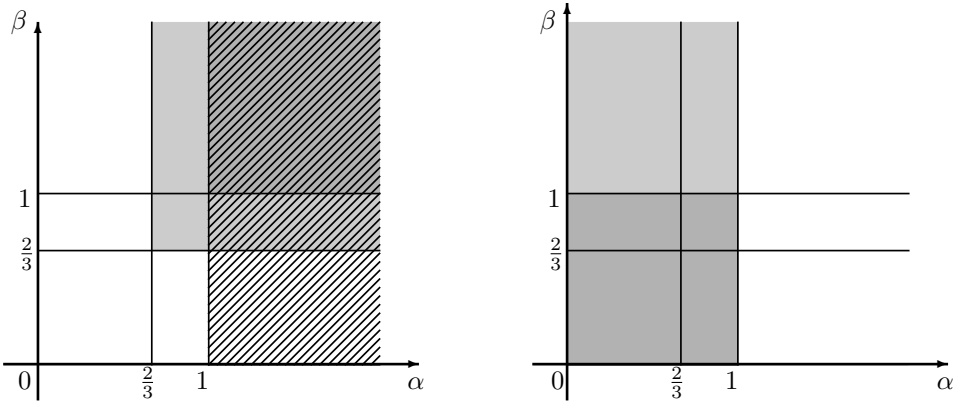

\section{\texorpdfstring{$\rmL^2$}{L2} estimates with energy based weights}\label{Sec:Coupling}

In this section, we consider a twisted $\rmL^2$ norm which plays the role of an entropy in the so-called \emph{DMS method} introduced by Dolbeault, Mouhot and Schmeiser in~\cite{MR2576899,MR3324910} and study the dissipation of the entropy. The method relies on four main assumptions: (H1) \emph{microscopic coercivity} (Poincar\'e inequality in the $v$-variable), (H2) \emph{macroscopic coercivity} (Poincar\'e inequality in the $x$-variable), (H3) \emph{parabolic macroscopic dynamics} (which is a compatibility condition of $\sfT$ with the projection onto the kernel of $\sfL$, and (H4) an assumption of boundedness on \emph{auxiliary operators}: see~\cite[Section~1.3]{MR3324910} for details. If $\alpha\ge1$ and $\beta\ge1$, the method directly applies (Section~\ref{Sec:loss}) and we only sketch the main steps of the proof. Using refinements based on weights built as functions of $E$, we extend the method to the range given by either $\alpha\ge1$ or $\min\{\alpha,\beta\}>2/3$.

\subsection{Entropy dissipation and loss of weights}\label{Sec:loss}

Let us start by recalling a useful \emph{weighted Poincar\'e inequality} associated to the probability measure
\[
\tildemeas\gamma(y):=\frac{e^{-\,\chi_{\gamma,\varepsilon}(y)}\dd y}{\int_\Rd{e^{-\,\chi_{\gamma,\varepsilon}(z)}\dd z}}\quad\mbox{with}\quad
\chi_{\gamma,\varepsilon}(y):=\frac{1-\varepsilon}\gamma\,\wangle y^\gamma
\]
that can be found in~\cite[Corollary~10]{MR4200970}.
\begin{lemma}\label{Lem:WeightedPoincare}{\rm~\cite{MR4200970}} For any $\gamma>0$ and $\varepsilon\in[0,1)$, there is a positive constant $\mathscr C_{\gamma,\varepsilon}$ such that
\be{Ineq:WeightedPoincare}
\int_\Rd|\nabla_yw|^2\tildemeas\gamma\ge\mathscr C_{\gamma,\varepsilon}\int_\Rd\frac{|w-\overline w|^2}{\wangle y^{2\,(1-\gamma)_+}}\tildemeas\gamma\quad\mbox{where}\quad\overline w=\int_\Rd w\tildemeas\gamma
\ee
for any $w\in\rmL^2\(\Rd,\wangle y^{-2\,(1-\gamma)_+}\tildemeas\gamma\)$ such that $\nabla_yw\in\rmL^2\(\Rd,\tildemeas\gamma\)$. \end{lemma}
Notice that the average $\overline w$ is computed with respect to the measure of the left-hand side if $\gamma\in(0,1)$, while~\eqref{Ineq:WeightedPoincare} coincides with the standard Poincar\'e inequality for the measure $e^{-\wangle y^\gamma}\,\kern-3pt\dd y$ if $\gamma\ge1$. If $\gamma\in(0,1)$, we already discover a loss of weight in the right-hand side of~\eqref{Ineq:WeightedPoincare}.

\medskip If $f$ is a solution to~\eqref{kFP}, using $\sfL f=\nabla_v\cdot\big(f_\star\,\nabla_v(f/f_\star)\big)$ and an integration by parts, we learn that
\[
\frac{\mathrm d}{\mathrm dt}\|f\|^2_{\rmL^2(\dd\mu_\star)}=-\,2\irdxvfstar{|\nabla_v(f/f_\star)|^2}
\]
where the right-hand side vanishes whenever $f$ is a \emph{local equilibrium}, \emph{i.e.},
\[
f(t,x,v)=\rho(t,x)\,F(v)\quad\mbox{where}\quad F(v)=\big(\textstyle\irdv{e^{-\psi}}\big)^{-1}\,e^{-\psi(v)}\,.
\]
The key idea in the \emph{DMS method} of~\cite{MR2576899,MR3324910} relies on a \emph{modified entropy} functional
\[
\sfH[f]:=\frac12\,\|f\|^2_{\rmL^2(\dd\mu_\star)}+\delta\,\bangle{\sfA f,f}_{\rmL^2(\dd\mu_\star)}\quad\mbox{where}\quad\sfA:=\big(1+(\sfT\sfPi)^*(\sfT\sfPi)\big)^{-1}(\sfT\sfPi)^*
\]
for $\delta>0$, small enough, in order to obtain a strict monotonicity unless $f$ is equal to the \emph{global equilibrium}
\[
f_\star(x,v)=\rho_\star(x)\,F(v)\quad\mbox{where}\quad\rho_\star(x)=\big(\textstyle\irdx{e^{-\phi}}\big)^{-1}\,e^{-\phi(x)}=\irdv{f_\star(x,v)}\,.
\]
Here we assume that the initial datum $f_0\ge0$ is such that \hbox{$\nrmxv{f_0}1=1$} and define
\[
\sfPi f:=\frac{f_\star}{\rho_\star}\,\rho_f=w_f\,F\quad\mbox{where}\quad w_f:=\frac{\rho_f}{\rho_\star}\quad\mbox{and}\quad\rho_f(t,x):=\irdv{f(t,x,v)}
\]
as the $\rmL^2(\dd\mu_\star)$ orthogonal projection onto the local equilibria. Such a functional is equivalent to the $\rmL^2(\dd\mu_\star)$ norm, \emph{i.e.}, $\sfH[\cdot]\asymp \norm{\cdot}^2_{\rmL^2(\dd\mu_\star)}$ if $\delta>0$ is taken small enough. and it is possible to provide some bound from below on the dissipation of entropy $\sfD[f]$, which goes as follows.
\begin{lemma}\label{lem:betterDiss}
For any $\alpha>0$, $\beta>0$, and $\delta>0$ small enough, there is a positive constant $\kappa$ such that, if $f\in\mathrm C^0\big(\R^+;\rmL^1(\dx\,\dv)\cap\rmL^2(\mathrm d\mu_\star)\big)$ solves~\eqref{kFP}, then
\be{DEstim}
-\,\frac{\mathrm d}{\mathrm dt}\sfH[f]=:\sfD[f]\ge\kappa \(\|(1-\sfPi)f\|_{\rmL^2\(\wv^{-2\,(1-\beta)_+} \dd\mu_\star\)}^2+\|\sfPi f\|_{\rmL^2\(\wx^{-2\,(1-\alpha)_+} \dd\mu_\star\)}^2\).
\ee
\end{lemma}
\begin{proof} From~\cite[Lemma~1]{MR4769957} we have that
\be{D1}
\sfD[f]\gtrsim \(\|(1-\sfPi)f\|_{\rmL^2\big(\wv^{-2\,(1-\beta)_+}\dd\mu_\star\big)}^2+\bangle{\sfA\sfT\sfPi f,\sfPi f}_{\rmL^2(\dd\mu_\star)}\)
\ee
as the two terms of the right-hand side in~\eqref{D1} dominate the other ones. Here we use the weighted Poincar\'e inequality~\eqref{Ineq:WeightedPoincare} with $y=v$, $\gamma=\beta$ and $\varepsilon=0$. We also learn from the proof of~\cite[Lemma~1]{MR4769957} that
\[
\bangle{\sfA\sfT\sfPi f,\sfPi f}_{\rmL^2(\dd\mu_\star)}=\sigma\int_\Rd|\nabla_xu|^2\,\rho_\star\dx+\sigma^2\int_\Rd\frac1{\rho_\star}\,|\nabla_x\cdot(\rho_\star\nabla_xu)|^2 \dx
\]
where $u$ is a solution to the elliptic equation
\be{Eq:u}
u-\frac{\sigma}{\rho_\star}\,\nabla_x\cdot\(\rho_\star\nabla_xu\)=w_f
\ee
with diffusion coefficient $\sigma=\frac{\irdv{|v|^2\,\bangle v^{2\,\beta-4}\,e^{-\frac1\beta\,\bangle v^\beta}}}{d\irdv{\,e^{-\frac1\beta\,\bangle v^\beta}}} $. By squaring~\eqref{Eq:u}, we obtain
\[
\(w_f\)^2=u^2 -2\,\sigma\,\frac u{\rho_\star}\,\nabla_x\cdot(\rho_\star\nabla_xu) + \frac{\sigma^2}{\rho_\star^2}\,|\nabla_x\cdot(\rho_\star\nabla_xu)|^2\,.
\]
Hence
\begin{align*}
\|\sfPi f&\|_{\sfL^2\big(\wx^{-2\,(1-\alpha)_+}\dd\mu_\star\big)}^2=\iint_{\RRd}\(w_f\)^2 f_\star\,\wx^{-2\,(1-\alpha)_+}\dx\dv\\
&=\int_\Rd u^2\,\rho_\star\,\wx^{-2\,(1-\alpha)_+} \dx - 2\,\sigma\int_\Rd u\,\nabla_x\cdot(\rho_\star\nabla_xu)\,\wx^{-2\,(1-\alpha)_+} \dx\\
&\qquad+\sigma^2\int_\Rd\frac1{\rho_\star}\,|\nabla_x\cdot(\rho_\star\nabla_xu)|^2\,\wx^{-2\,(1-\alpha)_+} \dx\\
&=\int_\Rd u^2\,\rho_\star\,\wx^{-2\,(1-\alpha)_+} \dx + 2\,\sigma\int_\Rd|\nabla_xu|^2\,\rho_\star\,\wx^{-2\,(1-\alpha)_+} \dx\\
&\qquad+ 2\,\sigma\int_\Rd u\,\nabla_xu \cdot \nabla_x \big(\wx^{-2\,(1-\alpha)_+}\big)\,\rho_\star\dx\\
&\qquad+\sigma^2\int_\Rd\frac1{\rho_\star}\,|\nabla_x\cdot(\rho_\star\nabla_xu)|^2\,\wx^{-2\,(1-\alpha)_+} \dx\\
&\le\int_\Rd u^2\,\rho_\star\,\wx^{-2\,(1-\alpha)_+} \dx + 2\,\sigma\int_\Rd|\nabla_xu|^2\,\rho_\star\dx\\
&\qquad- \sigma\int_\Rd u^2\,\nabla_x\cdot\(\rho_\star\nabla_x \big(\wx^{-2\,(1-\alpha)_+}\big)\)\dx+\sigma^2\int_\Rd\frac1{\rho_\star}\,|\nabla_x\cdot(\rho_\star\nabla_xu)|^2 \dx\,.
\end{align*}
Using $\left|\nabla_x\cdot\(\rho_\star\nabla_x \big(\wx^{-2\,(1-\alpha)_+}\big)\)\right|\lesssim \wx^{-2\,(1-\alpha)_+}\rho_\star$, we obtain
\begin{align*}
\|\sfPi f\|_{\sfL^2\big(\wx^{-2\,(1-\alpha)_+}\dd\mu_\star\big)}^2\lesssim&\;\int_\Rd u^2\,\rho_\star\,\wx^{-2\,(1-\alpha)_+} \dx\\
&\qquad+ \int_\Rd|\nabla_xu|^2\,\rho_\star\dx +\int_\Rd\frac1{\rho_\star}\,|\nabla_x\cdot(\rho_\star\nabla_xu)|^2 \dx\,,\end{align*}
that is,
\be{D2}
\|\sfPi f\|_{\sfL^2\big(\wx^{-2\,(1-\alpha)_+}\dd\mu_\star\big)}^2\lesssim\bangle{\sfA\sfT\sfPi f,\sfPi f}_{\rmL^2(\dd\mu_\star)}
\ee
where, in the last step, we used the weighted Poincaré inequality~\eqref{Ineq:WeightedPoincare} with $y=x$, $\gamma=\alpha$ and $\varepsilon=0$.
\end{proof}

At this point, if $\min\{\alpha,\beta\}\ge1$, it is worth to notice that the proof of~\eqref{Ineq:Exp} is complete because
\[
-\,\frac{\mathrm d}{\mathrm dt}\sfH[f-f_\star]=\sfD[f-f_\star]\gtrsim\sfH[f-f_\star]
\]
by Lemma~\ref{lem:betterDiss}. A Gr\"onwall estimate then shows that, for some explicit $\lambda>0$,
\[
\norm{f-f_\star}_{\rmL^2(\mathrm d\mu_\star)}^2\asymp\sfH[f-f_\star]\le\sfH[f_0-f_\star]\,e^{-\lambda\,t}\asymp\norm{f_0-f_\star}_{\rmL^2(\mathrm d\mu_\star)}^2\,e^{-\lambda\,t}\quad\forall\,t\ge0\,.
\]
If $\min\{\alpha,\beta\}<1$, in order to control $\sfH[f]$ by $\sfD[f]$ using~\eqref{D1} and~\eqref{D2}, moments in $\wx$ and $\wv$ and weighted $\rmL^2$ norms are needed to compensate for the weights $\wx^{-2\,(1-\alpha)_+}$ if $\alpha\in(0,1)$ and $\wv^{-2\,(1-\beta)_+}$ if $\beta\in(0,1)$.

The \emph{energy moment} of order $k\in\R^+$ (in fact a weighted $\rmL^2$ norm) of a solution $f$ of~\eqref{kFP} is defined by
\[
M_k(t):=\irdmustar{|f(t,x,v)|^2\,E(x,v)^k}\,.
\]
{}From here on $'$ denotes the $t$-derivative.
\begin{lemma}\label{Lem:H} Let $(\alpha,\beta)\in(0,+\infty)^2$, $k\in(0,+\infty)$. If $f\in\mathrm C^0\big(\R^+;\rmL^1(\dx\,\dv)\cap\rmL^2(\mathrm d\mu_\star)\big)$ solves~\eqref{kFP} and
\be{Mk}
M_k(t)\le c_k\,M_k(0)\quad\forall\,t\ge0
\ee
for some constant $c_k>0$, then there exists a constant $C_k>0$ such that
\be{Ht}
\sfH[f(t,\cdot,\cdot)]=:H(t)\le\(H(0)^{-\frac2\zeta}+C_k\,M_k(0)^{-\frac2\zeta}\,t\)^{-\zeta/2}\quad\mbox{with}\quad\zeta=\min\big\{\tfrac k{1-\alpha},\tfrac k{1-\beta}\big\}\,.
\ee
\end{lemma}
\begin{proof} In order to compare $H'$ with $H$ using~\eqref{DEstim}, we introduce weighted $\rmL^2$ norms. Let $b=\tfrac{k\,\beta}{k\,\beta + 2\,(1-\beta)_+}$. Since $1-\sfPi$ is a projection, by H\"older's inequality we have
\begin{align*}
\norm{(1-\sfPi)f}^2_{\rmL^2(\dd\mu_\star)} &\le\|(1-\sfPi)f\|^{2\,b}_{\sfL^2\big(\wv^{-2\,(1-\beta)_+}\dd\mu_\star\big)}\,\|(1-\sfPi)f\|^{2\,(1-b)}_{\sfL^2\big(\wv^{k\,\beta}\dd\mu_\star\big)}\\
&\lesssim \|(1-\sfPi)f\|_{\sfL^2\big(\wv^{-2\,(1-\beta)_+}\dd\mu_\star\big)}^{2\,b}\,M_k^{1-b}
\end{align*}
so that
\[
\|(1-\sfPi)f\|_{\sfL^2\big(\wv^{-2\,(1-\beta)_+}\dd\mu_\star\big)}^2 \gtrsim\norm{(1-\sfPi)f}^{\tfrac2b}_{\rmL^2(\dd\mu_\star)}\,M_k^{-\tfrac{1-b}b}\,.
\]
Similarly, by setting $a=\frac{k\,\alpha}{k\,\alpha+2\,(1-\alpha)_+}$, we obtain
\[
\|\sfPi f\|_{\sfL^2\big(\wx^{-2\,(1-\alpha)_+}\dd\mu_\star\big)}^2 \gtrsim \norm{\sfPi f}^{\tfrac2a}_{\rmL^2(\dd\mu_\star)}\,M_k^{-\tfrac{1-a}{a}}\,.
\]
As a consequence, $\eta=\min\{a,b\}\le 1$ is such that
\be{TwoThird}
\frac1\eta<1+\frac1k\quad\mbox{if}\quad(\alpha,\beta)\in(2/3,+\infty)^2\,.
\ee
On the other hand, using the inequality $s^{1/a}\ge s^{1/\eta}$ and $s^{1/b}\ge s^{1/\eta}$ for every $s\in[0,1]$ proves that
\begin{align*}
\|(1-\sfPi)f\|_{\sfL^2\big(\wv^{-2\,(1-\beta)_+}\dd\mu_\star\big)}^2 &\gtrsim\Bigg( \frac{\norm{(1-\sfPi)f}^2_{\rmL^2(\dd\mu_\star)} }{\norm f^2_{\rmL^2(\dd\mu_\star)}} \Bigg)^{\tfrac1b}\Bigg(\frac{\norm f^2_{\rmL^2(\dd\mu_\star)}}{M_k}\Bigg)^{\frac1b}\,M_k\\
&\gtrsim\Bigg( \frac{\norm{(1-\sfPi)f}^2_{\rmL^2(\dd\mu_\star)} }{\norm f^2_{\rmL^2(\dd\mu_\star)}} \Bigg)^{\tfrac1\eta}\Bigg(\frac{\norm f^2_{\rmL^2(\dd\mu_\star)}}{M_k}\Bigg)^{\frac1\eta}\,M_k\\
& \gtrsim\norm{(1-\sfPi)f}^{\tfrac2\eta}_{\rmL^2(\dd\mu_\star)}\,M_k^{-\tfrac{1-\eta}\eta}\,,\\
\|\sfPi f\|_{\sfL^2\big(\wx^{-2\,(1-\alpha)_+}\dd\mu_\star\big)}^2 &\gtrsim \norm{\sfPi f}^{\tfrac2\eta}_{\rmL^2(\dd\mu_\star)}\,M_k^{-\tfrac{1-\eta}\eta}\,.
\end{align*}
Putting the estimates together we obtain
\begin{align*}
H'\lesssim -\,M_k^{-\tfrac{1-\eta}\eta}\,\Bigg(\norm{(1-\sfPi) f}^{\tfrac2\eta}_{\rmL^2(\dd\mu_\star)} + \norm{\sfPi f}^{\tfrac2\eta}_{\rmL^2(\dd\mu_\star)} \Bigg)\lesssim -\,M_k^{-\tfrac{1-\eta}\eta}\,\norm f^{\tfrac2\eta}_{\rmL^2(\dd\mu_\star)}
\end{align*}
and conclude that
\be{dHdt}
H'\lesssim -\, M_k^{-\tfrac{1-\eta}\eta}H^{\tfrac1\eta}\,.
\ee
\end{proof}

\subsection{Entropy dissipation and energy moments}\label{Sec:Energy}

Using a \emph{coupling} strategy, we can combine the estimate of Lemma~\ref{Lem:H} with an estimate on $M_k(t)$ and extend the results of Section~\ref{Sec:loss} to the range $\frac23<\min\{\alpha,\beta\}<1$.
\begin{proposition}\label{eq:alphabeta23} Assume that $k\in(0,+\infty)$, $\min\{\alpha,\beta\}>2/3$. Then there exists a constant $c_k>0$ such that, if $f\in\mathrm C^0\big(\R^+;\rmL^1(\dx\,\dv)\cap\rmL^2(\mathrm d\mu_\star)\big)$ solves~\eqref{kFP}, then~\eqref{Mk} holds. As a consequence, if $2/3<\min\{\alpha,\beta\}<1$, there exists a constant $C_k>0$, which is independent of $f$, such that~\eqref{Ht} holds.
\end{proposition}
\begin{proof} We build a system of differential equations in order to estimate simultaneously $H(t):=\sfH[f(t,\cdot,\cdot)]$ and $M_k(t)$.

\step1{Growth estimate on $M_k$.} We need to compute the time derivative of $M_k$. Since $\sfT$ is anti-symmetric and such that $\sfT E=0$ we have
\begin{multline*}
M_k'=-\,2\irdxvfstar{\kern-6pt|\nabla_v(f/f_\star)|^2\,E^k}+\irdxv{\kern-6pt(f/f_\star)^2\,\nabla_v\cdot\(f_\star\,\nabla_vE^k\)}\\
\le\irdxv{\kern-6pt(f/f_\star)^2\,\nabla_v\cdot\(f_\star\,\nabla_vE^k\)}\,.
\end{multline*}
Moreover, using
\begin{multline*}
f_\star^{-1}\,\nabla_v\cdot\(f_\star\,\nabla_vE^k\)=k\,E^{k-1}\(\Delta_vE-|\nabla_vE|^2+(k-1)\,\frac{|\nabla_vE|^2}E\)\\
\le k\,(2\,d-2+\beta\,k)\,E^{k-1}
\end{multline*}
and H\"older's inequality, we obtain
\be{dMkdt}
M_k'\le k\,(2\,d-2+\beta\,k)\,M_k^{\frac{k-1}k}\,\norm f_{\rmL^2(\dd\mu_\star)}^{\frac2k}\le c\,M_k^{\frac{k-1}k}\,H^{\frac1k}
\ee
for some $c>0$.

\step2{Joint estimates.} Taking into account~\eqref{dHdt} and~\eqref{dMkdt}, we notice that
\begin{align*}
\frac{\mathrm d}{\mathrm dt}\(c\,H^{1+\frac1k-\frac1\eta} + C\,M_k^{1+\frac1k-\frac1\eta}\)&=\(1+\frac1k-\frac1\eta\)\(c\,H^{\frac1k-\frac1\eta}\frac{\mathrm d}{\mathrm dt}H+ C\,M_k^{\frac1k-\frac1\eta}M_k'\)\\
&\le c\,C\(1+\frac1k-\frac1\eta\) \( -\,H^{\frac1k}M_k^{1-\frac1\eta}+H^{\frac1k} M_k^{1-\frac1\eta}\)\le 0
\end{align*}
because $1+1/k-1/\eta>0$ according to~\eqref{TwoThird}. As a consequence, we obtain
\[
c\,H(t)^{1+\frac1k-\frac1\eta} + C\,M_k(t)^{1+\frac1k-\frac1\eta}\le c\,H(0)^{1+\frac1k-\frac1\eta} + C\,M_k(0)^{1+\frac1k-\frac1\eta}\lesssim M_k(0)^{1+\frac1k-\frac1\eta}
\]
for all $t\ge 0$, which proves that $M_k(t)\le c_k\,M_k(0)$ for some constant $c_k$. Inserting this estimate in~\eqref{dHdt} and integrating the differential inequality on $(0,t)$ completes the proof with $C_k=2\,C\,\zeta^{-1}\,c_k^{-\zeta/2}$ and $\zeta=2\,\eta/(1-\eta)$.
\end{proof}

As we rely on H\"older's inequalities in the proof of Proposition~\ref{eq:alphabeta23}, we can only expect a power law rate of decay if $2/3<\min\{\alpha,\beta\}<1$. However, with just Step~1 of this proof, we already have a complete proof of Proposition~\ref{prop:rates} for $\ell=k$ and $(\alpha,\beta)\in(0,+\infty)^2$, under the assumption that $\mathscr K[f]$ is finite, which is in fact the difficult point. The strategy of~\cite{MR4769957} was to assume that $f_0\lesssim f_\star$, which is a very restrictive condition. In Proposition~\ref{eq:alphabeta23}, we found that $2/3<\min\{\alpha,\beta\}<1$ and $M_k(0)<+\infty$ is enough. Notice that the estimate on $M_k(t)$ of Step~1 is valid for any $(\alpha,\beta)\in(0,+\infty)^2$.

\subsection{Entropy dissipation and stronger weight}\label{Sec:modifiedDMS}

In the range $\alpha\in[1,+\infty)$ and $\beta\in(0,1)$, we first obtain the propagation of energy moments for solutions in a very small weighted $\rmL^2$ space, for weights given by a rapidly increasing function of $E$, later extend the energy moments estimates to the natural $\rmL^2$ space by interpolation.

\subsubsection{Exponential energy weights}\label{Sec:ExponentialWeights}

Let us take $\varepsilon\in(0,1)$ and consider the weighted space $\cH_\varepsilon:=\rmL^2\(\dd\mu_\varepsilon\)$. Here we adopt the notation
\[
\mathrm d\mu_\varepsilon:=f_\star^{-(1+\varepsilon)}\dx\dv\quad\mbox{and}\quad\mathrm d\widetilde\mu_\varepsilon:=f_\star^{1-\varepsilon}\dx\dv\,,
\]
so that $\mathrm d\mu_\varepsilon$ and $\mathrm d\widetilde\mu_\varepsilon$ coincide respectively with $\mathrm d\mu_\star$ and $\mathrm d\widetilde\mu_\star$ if $\varepsilon=0$. Since $f_\star=Z^{-1}\,e^{-E}$ so that $f_\star\in\cH_\varepsilon$, we may expect that the $\cH_\varepsilon$-norm is also preserved for solutions along the evolution, at least for $\varepsilon>0$ sufficiently small. The main result of this section is the following.
\begin{proposition}\label{prop:energyMomDMS}
Let $\alpha\ge 1$ and $\beta>0$. There exists $\varepsilon>0$ small enough such that, for any $f_0\in\cH_\varepsilon$, the semi-group $S_\cL$ satisfies
\be{eq:expBounded}
\norm{S_\cL(t)}_{\cH_\varepsilon\to\cH_\varepsilon}\lesssim 1\quad\forall\,t\ge0\,.
\ee
As a consequence, for any $k\ge 0$, we have
\be{eq:polyBounded}
\norm{S_\cL(t)}_{\rmL^2\(E^k\dd\mu_\star\)\to \rmL^2\(E^k\dd\mu_\star\)}\lesssim 1\,.
\ee
\end{proposition}
An important feature of the weight $e^{\varepsilon E}$ is that $\sfT$ is skew symmetric as an operator acting on $\cH_\varepsilon$. The factorization property
\[
e^{\varepsilon E}f_\star=\frac1Z\,e^{-\,(1-\varepsilon)\,E}=\frac1Z\,e^{-\,(1-\varepsilon)\,\psi}\,e^{-\,(1-\varepsilon)\,\phi}
\]
allows us to use Inequality~\eqref{Ineq:WeightedPoincare} (with $\varepsilon\neq0$) in $x$ and $v$ as in the proof of Lemma~\ref{lem:betterDiss}, with an optimal constant $C_{\gamma,\varepsilon}$ which depends on $\varepsilon$. Notice that the average $\overline w=\int_\Rd w\tildemeas\gamma$ also depends on $\varepsilon$.

With exponential energy weights, our strategy goes as follows. Compared to the techniques of Section~\ref{Sec:loss} whenever $\min\{\alpha,\beta\}\ge1$, the proof of~\eqref{eq:expBounded} relies on \emph{a modified DMS method} with the Poincar\'e inequality replaced by the \emph{weighted Poincaré inequality}~\eqref{Ineq:WeightedPoincare}: see Section~\ref{Sec:Dissipation}. Estimate~\eqref{eq:polyBounded} is a consequence of an interpolation result of~\cite{BL76}: see Section~\ref{Sec:interp}.

\subsubsection{A modified DMS method}\label{Sec:DMSmodified}

Key issues are the replacement of the Poincar\'e inequality in~\cite{MR3324910} (\emph{microscopic coercivity}) by the \emph{weighted Poincaré inequality}~\eqref{Ineq:WeightedPoincare}. Other assumptions of the \emph{DMS method} as described in~\cite[Section~1.3]{MR3324910} are unchanged, except that by replacing $\rmL^2\(\dd\mu_\star\)$ by $\cH_\varepsilon$, the operator $\sfL$ is not anymore self-adjoint.

On the Hilbert space $\cH_\varepsilon$, let us define the orthogonal projection
\[
\sfPi_\varepsilon\,f=\frac{\rho_f^{(\varepsilon)}}{\rho_\star^{(\varepsilon)}}\,f_\star\quad\mbox{where}\quad\rho_f^{(\varepsilon)}:=\int_\Rd f\,e^{\varepsilon E}\dv\quad\mbox{and}\quad\rho_\star^{(\varepsilon)}:=\int_\Rd f_\star^{1-\varepsilon}\dv\,.
\]
Let us define
\[
\dd\nu_\varepsilon:=\frac1{\rho_\star^{(\varepsilon)}}\dx\quad\mbox{and}\quad\mathrm d\widetilde\nu_\varepsilon:=\rho_\star^{(\varepsilon)}\dx\,.
\]
We shall use the notation $\bangle{\cdot,\cdot}_{\cH_\varepsilon}$ for the canonical scalar product. From here on,
\emph{adjoint} operators are defined with respect to $\bangle{\cdot,\cdot}_{\cH_\varepsilon}$. The operators $\sfPi_\varepsilon$ and $\sfT$ acting on $\cH_\varepsilon$ have the standard properties: $\sfPi_\varepsilon^*=\sfPi_\varepsilon=\sfPi_\varepsilon^2$, and
\[
\bangle{\sfPi_\varepsilon\,f,g}_{\cH_\varepsilon}=\int_\Rd\frac{\rho_f^{(\varepsilon)} \rho_g^{(\varepsilon)} }{\rho_\star^{(\varepsilon)}}\dx=\bangle{ f,\sfPi_\varepsilon g}_{\cH_\varepsilon}\quad\forall\,f,\,g\in\cH_\varepsilon\,.
\]
On $\cH_\varepsilon$, $\sfT^*=-\,\sfT$ is skew-symmetry and the (H3) \emph{parabolic macroscopic dynamics} property
\be{eq:TPi_adj}
\sfPi_\varepsilon\sfT\sfPi_\varepsilon=0
\ee
holds. The collision operator $\sfL$ is not self-adjoint in $\cH_\varepsilon$. A direct computation shows that
\be{eq:L*}
\sfL^* f=\sfL f + 2\,\varepsilon\,\nabla_v \psi\cdot\nabla_v (f/f_\star)\,f_\star + \varepsilon\,f\,\Big(\Delta_v \psi -(1-\varepsilon)\,\abs{\nabla_v \psi}^2\Big)\,.
\ee
This is another significant difference with the standard method of~\cite{MR3324910}, which has however not much consequences on the decay of $\norm f^2_{\cH_\varepsilon}$. If $f$ solves~\eqref{kFP}, we have
\begin{align*}
&\frac{\mathrm d}{\mathrm dt}\norm f^2_{\cH_\varepsilon}=2\iint_\RRd f\,\sfL f\dd\mu_\varepsilon - \iint_\RRd \sfT\big(f^2\big)\dd\mu_\varepsilon\\
&=2\iint_\RRd f\,\nabla_v \cdot\big(f_\star\nabla_v (f/f_\star)\big)\dd\mu_\varepsilon\\
&=-\,2\iint_\RRd\big|\nabla_v (f/f_\star)\big|^2\dd\widetilde\mu_\varepsilon+\iint_\RRd (f/f_\star)^2\,\nabla_v \cdot\(f_\star\nabla_v\big(e^{\varepsilon E}\big)\)\dx\dv\\
&=-\,2\iint_\RRd\big|\nabla_v(f/f_\star)\big|^2\dd\widetilde\mu_\varepsilon+\varepsilon \iint_\RRd (f/f_\star)^2\,\Big(\Delta_v \psi -(1-\varepsilon)\,\abs{\nabla_v \psi}^2\Big) \dd\widetilde\mu_\varepsilon\,.
\end{align*}
Using the weighted Poincaré inequality~\eqref{Ineq:WeightedPoincare} with $y=v$, $\gamma=\beta$ and $\varepsilon\in(0,1)$, and direct estimates on $\psi$, we obtain
\be{eps:v-coercivity}
\frac12\,\frac{\mathrm d}{\mathrm dt}\norm f^2_{\cH_\varepsilon}=\bangle{f,\sfL f}_{\cH_\varepsilon}\lesssim -\,\norm{(1-\sfPi_\varepsilon)f}^2_{\cH_\varepsilon} +\varepsilon\,\norm f^2_{\rmL^2\big(\wv^{-2\,(1-\beta)_+}\dd\mu_\varepsilon\big)}
\ee
which replaces the \emph{microscopic coercivity} (H1) property of~\cite{MR3324910}.

Let us introduce the \emph{entropy}
\[
\sfH_\varepsilon[f]:=\frac12\,\|f\|^2_{\cH_\varepsilon}+\delta\,\bangle{\sfA_\varepsilon\,f,f}_{\cH_\varepsilon}\quad\mbox{with}\quad\sfA_\varepsilon:=\big(1+(\sfT\sfPi_\varepsilon)^*(\sfT\sfPi_\varepsilon)\big)^{-1}(\sfT\sfPi_\varepsilon)^*
\]
for some $\delta\in(0,2)$ to be chosen, and the \emph{entropy dissipation}
\be{Df}\begin{aligned}
\sfD_\varepsilon[f] :=&\,-\,\bangle{\sfL f, f}_{\cH_\varepsilon}+ \delta\,\bangle{\sfA_\varepsilon\sfT\sfPi_\varepsilon\,f,\sfPi_\varepsilon\,f}_{\cH_\varepsilon} - \delta\,\Big(\bangle{\sfT\sfA_\varepsilon\,f,(1-\sfPi_\varepsilon)f}_{\cH_\varepsilon}\\
&\,-\,\bangle{\sfA_\varepsilon\sfT(1-\sfPi_\varepsilon)f,\sfPi_\varepsilon\,f}_{\cH_\varepsilon} + \bangle{\sfA_\varepsilon\sfL f,f}_{\cH_\varepsilon} + \bangle{\sfA_\varepsilon\,f,\sfL f}_{\cH_\varepsilon}\Big)\,.
\end{aligned}
\ee
Using~\eqref{eq:TPi_adj}, $\sfD_\varepsilon$ is such that
\be{Df2}
\sfD_\varepsilon[f]=-\,\frac{\mathrm d}{\mathrm dt}\sfH_\varepsilon[f]
\ee
for any solution of~\eqref{kFP}. Next we proceed as in Section~\ref{Sec:loss}, except that we have to keep track of the additional terms which appear in~\eqref{eq:L*} and~\eqref{eps:v-coercivity}, and provide more details because various estimates are more original than in Sections~\ref{Sec:loss} and~\ref{Sec:Energy}.

\subsubsection{Basic observations}\label{Sec:Basic}

We have
\begin{align*}
\sfT\sfPi_\varepsilon\,f=f_\star\sfT\(\frac{\rho_f^{(\varepsilon)}}{\rho_\star^{(\varepsilon)}}\)=f_\star\,\nabla_v \psi \cdot\nabla_x \(\frac{\rho_f^{(\varepsilon)}}{\rho_\star^{(\varepsilon)}}\)
\end{align*}
and
\begin{align*}
\sfPi_\varepsilon \sfT f &=\frac{f_\star}{\rho_\star^{(\varepsilon)}}\int_\Rd\Big(\nabla_v \psi\cdot\nabla_x f - \nabla_x \phi\cdot\nabla_v f\Big)\,e^{\varepsilon E} \dv\\
&=\frac{f_\star}{\rho_\star^{(\varepsilon)}}\int_\Rd\Big(\nabla_v \psi\cdot\nabla_x f +\varepsilon\,\nabla_x \phi\cdot\nabla_v \psi\,f\Big)\,e^{\varepsilon E} \dv=\frac{f_\star}{\rho_\star^{(\varepsilon)}}\,\nabla_x \cdot\int_\Rd\nabla_v \psi\,f\,e^{\varepsilon E} \dv\,.
\end{align*}
Hence
\[
(\sfT\sfPi_\varepsilon)^*(\sfT\sfPi_\varepsilon) f=-\,(\sfPi_\varepsilon\sfT)\,(\sfT\sfPi_\varepsilon)f=-\,\frac1{\rho_\star^{(\varepsilon)}}\,\nabla_x \cdot \(\mathscr D_\varepsilon\,\nabla_x \({\rho_f^{(\varepsilon)}}/{\rho_\star^{(\varepsilon)}}\)\)f_\star\,,
\]
where $\mathscr D_\varepsilon:=\int_\Rd\nabla_v \psi\otimes\nabla_v \psi \,e^{\varepsilon E}\,f_\star\dv$. Using the radial symmetry of $\psi$, we obtain
\begin{align*}
(\sfT\sfPi_\varepsilon)^*(\sfT\sfPi_\varepsilon) f=-\,\frac{\sigma_\varepsilon}{\rho_\star^{(\varepsilon)}}\,\nabla_x \cdot \(\rho_\star^{(\varepsilon)}\,\nabla_x \({\rho_f^{(\varepsilon)}}/{\rho_\star^{(\varepsilon)}}\)\) f_\star
\end{align*}
where the diffusion coefficient
\[
\sigma_\varepsilon :=\frac{\int_\Rd\abs{v}^2\,\wv^{2\,\beta-4}\,e^{-\,(1-\varepsilon)\,\psi(v)}\dv}{d\int_\Rd e^{-\,(1-\varepsilon)\,\psi(v)}\dv}
\]
does not vanish as $\varepsilon\to0_+$. The next observation is that the quadratic form determines a norm which is equivalent to the standard norm of $\cH_\varepsilon$.
\begin{lemma}\label{Lem:Equivalence-eps}{\rm \cite{MR4769957,MR3324910}} Assume that $\delta\in(0,2)$. For any $f\in\cH_\varepsilon$, we have
\[
\frac{2-\delta}4\,\norm f^2_{\cH_\varepsilon}\le\sfH_\varepsilon[f]\le\frac{2+\delta}4\,\norm f^2_{\cH_\varepsilon}\,.
\]
\end{lemma}
\begin{proof} This result is well known and we claim no originality but give a sort proof for completeness. Let us consider the function $g=\sfA_\varepsilon\,f$, that is, the solution
\[
g + (\sfT\sfPi_\varepsilon)^*\sfT\sfPi_\varepsilon g=(\sfT\sfPi_\varepsilon)^*f\,.
\]
We may notice that $\sfPi_\varepsilon g=g$ and
\begin{align*}
\bangle{\sfT\sfA_\varepsilon\,f, f}_{\cH_\varepsilon} &=\bangle{\sfT \sfPi_\varepsilon g, f}_{\cH_\varepsilon}=\bangle{ g, (\sfT \sfPi_\varepsilon)^*f}_{\cH_\varepsilon}\\
&=\norm{g}^2_{\cH_\varepsilon} + \norm{\sfT\sfPi_\varepsilon g}^2_{\cH_\varepsilon}=\norm{\sfA_\varepsilon\,f}^2_{\cH_\varepsilon} + \norm{\sfT\sfA_\varepsilon\,f}^2_{\cH_\varepsilon}\,.
\end{align*}
On the other hand, using~\eqref{eq:TPi_adj}, we have
\[
\bangle{\sfT\sfA_\varepsilon\,f, f}_{\cH_\varepsilon}\le\norm{\sfT\sfA_\varepsilon\,f}_{\cH_\varepsilon}\norm{(1-\sfPi_\varepsilon)f}_{\cH_\varepsilon}\le\frac1{2\,\mu}\,\norm{\sfT\sfA_\varepsilon\,f}_{\cH_\varepsilon}^2 + \frac{\mu}2\,\norm{(1-\sfPi_\varepsilon)f}_{\cH_\varepsilon}^2
\]
Applied with $\mu=\tfrac12$ and $\mu=1$, we obtain
\be{eq:TAbdd}
\norm{\sfA_\varepsilon\,f}_{\cH_\varepsilon}\le\frac12\,\norm{(1-\sfPi_\varepsilon)f}_{\cH_\varepsilon}\quad\mbox{and}\quad\norm{\sfT\sfA_\varepsilon\,f}_{\cH_\varepsilon} \le\norm{(1-\sfPi_\varepsilon)f}_{\cH_\varepsilon}\,,
\ee
which completes the proof.
\end{proof}

\subsubsection{Dissipation of the entropy}\label{Sec:Dissipation}

The \emph{entropy dissipation} given by~\eqref{Df} can be estimated as follows.
\begin{lemma}\label{lem:Dissipation1}
For any $\beta>0$ and $\alpha\ge1$, there is $\kappa>0$ such that, for any $f\in\cH_\varepsilon$,
\begin{multline*}
\sfD_\varepsilon[f]+\varepsilon\,\norm f^2_{\rmL^2\big(\wv^{-2\,(1-\beta)_+}\dd\mu_\varepsilon\big)}\\
\ge\kappa \(\|(1-\sfPi_\varepsilon)f\|_{\rmL^2\big(\wv^{-2\,(1-\beta)_+}\dd\mu_\varepsilon\big)}^2+\bangle{\sfA_\varepsilon\sfT\sfPi_\varepsilon\,f,\sfPi_\varepsilon\,f}_{\cH_\varepsilon}\).
\end{multline*}
\end{lemma}
\begin{proof}~We have to estimate each of the terms of $\sfD_\varepsilon[f]$. We decompose the proof in six steps.

\Step{}Consider the function $u=u(x)$ defined by $u\,f_\star=\big(1+ (\sfT\sfPi_\varepsilon)^*(\sfT\sfPi_\varepsilon)\big)^{-1}\sfPi_\varepsilon\,f$, that is, the solution of
\be{eq:u}
\rho_\star^{(\varepsilon)}\,u-\sigma_\varepsilon\,\nabla_x \cdot \(\rho_\star^{(\varepsilon)}\,\nabla_xu\)=\rho_f^{(\varepsilon)}\,.
\ee
We have
\begin{align*}
\sfA_\varepsilon\sfT\sfPi_\varepsilon\,f&=\big(1+(\sfT\sfPi_\varepsilon)^\ast(\sfT\sfPi_\varepsilon)\big)^{-1}(\sfT\sfPi_\varepsilon)^\ast(\sfT\sfPi_\varepsilon)f\\
&=\sfPi_\varepsilon\,f-\big(1+(\sfT\sfPi_\varepsilon)^\ast(\sfT\sfPi_\varepsilon)\big)^{-1}\sfPi_\varepsilon\,f=\sfPi_\varepsilon\,f-u\,f_\star\\
&=-\,\frac{\sigma_\varepsilon}{\rho_\star^{(\varepsilon)}}\,\nabla_x \cdot \(\rho_\star^{(\varepsilon)}\,\nabla_xu\)\,f_\star
\end{align*}
After integration, we obtain
\[
\bangle{\sfA_\varepsilon\sfT\sfPi_\varepsilon\,f,\sfPi_\varepsilon\,f}_{\cH_\varepsilon}=-\,\sigma_\varepsilon\irdx{\nabla_x\cdot\big(\rho_\star^{(\varepsilon)}\,\nabla_xu\big)\(u-\frac{\sigma_\varepsilon}{\rho_\star^{(\varepsilon)}}\,\nabla_x\cdot\big(\rho_\star^{(\varepsilon)}\,\nabla_xu\big)\)}\,,
\]
which, after an integration by parts, proves that
\be{eq:ATPi}
\bangle{\sfA_\varepsilon\sfT\sfPi_\varepsilon\,f,\sfPi_\varepsilon\,f}_{\cH_\varepsilon}=\sigma_\varepsilon\,\norm{\nabla_xu}^2_{\rmL^2(\dd\widetilde\nu_\varepsilon)} + \sigma_\varepsilon^2\,\big\|\nabla_x \cdot(\rho_\star^{(\varepsilon)}\nabla_xu)\big\|^2_{\rmL^2(\dd\nu_\varepsilon)}
\ee
where $\dd\widetilde\nu_\varepsilon=\rho_\star^{(\varepsilon)}\dx$ and $\dd\nu_\varepsilon=\big(\rho_\star^{(\varepsilon)}\big)^{-1}\dx$.

\Step{}We claim that
\be{eq:TAeps}
\abs{\bangle{\sfT\sfA_\varepsilon\,f,(1-\sfPi_\varepsilon) f}_{\cH_\varepsilon}}\lesssim \norm{(1-\sfPi_\varepsilon)f}^2_{\rmL^2\big(\wv^{-2\,(1-\beta)_+}\dd\mu_\varepsilon\big)}\,.
\ee
If $\beta\ge 1$, we can directly use~\eqref{eq:TAbdd}. If $\beta \in(0,1)$, consider the function $w=w(x)$ implicitly defined by $w\,f_\star=\sfA_\varepsilon\,f$, that is, the solution of
\be{eq:w}
w-\frac{\sigma_\varepsilon}{\rho_\star^{(\varepsilon)}}\,\nabla_x\cdot(\rho_\star^{(\varepsilon)}\,\nabla_xw)=-\,\sfPi_\varepsilon\sfT f=-\,\frac1{\rho_\star^{(\varepsilon)}}\nabla_x\cdot\irdv{\nabla_v \psi\,f\,e^{\varepsilon E}}\,.
\ee
Testing with $w\,\rho_\star^{(\varepsilon)}$ we obtain
\begin{align*}
\sigma_\varepsilon\int_{\R^d}|\nabla_xw|^2\dd\widetilde\nu_\varepsilon\le&\int_{\R^d}|w|^2\dd\widetilde\nu_\varepsilon+\sigma_\varepsilon\int_{\R^d}|\nabla_xw|^2\dd\widetilde\nu_\varepsilon\\
&=\irdx{\nabla_xw\cdot\(\irdv{\nabla_v \psi\,f\,e^{\varepsilon E}}\)}\\
&=\irdx{\nabla_xw\cdot\(\irdv{\nabla_v\psi\,(1-\sfPi_\varepsilon)f\,e^{\varepsilon E}}\)}\,.
\end{align*}
By applying the Cauchy-Schwarz inequality and squaring, we obtain
\begin{align*}
\sigma_\varepsilon^2&\int_\Rd|\nabla_xw|^2\dd\widetilde\nu_\varepsilon\\
&\le\int_\Rd\(\irdv{\nabla_v \psi\,(1-\sfPi)f\,e^{\varepsilon E}}\)^2\dd\nu_\varepsilon\\
&\le\int_\Rd\(\irdv{\abs{(1-\sfPi_\varepsilon)f}^2\wv^{-2\,(1-\beta)}\,e^{\varepsilon E}\,f_\star^{-1}}\)\(\irdv{e^{\varepsilon E}f_\star}\)\dd\nu_\varepsilon\\
&\quad=\|(1-\sfPi_\varepsilon)f\|^2_{\rmL^2\big(\wv^{-2\,(1-\beta)}\dd\mu_\varepsilon\big)}\,.
\end{align*}
Altogether, we prove
\begin{align*}
&\|\sfT\sfA_\varepsilon\,f\,\|^2_{\rmL^2(\wv^{2\,(1-\beta)_+}\dd\mu_\varepsilon)}=\|\sfT(w\,f_\star)\|^2_{\rmL^2(\wv^{2\,(1-\beta)_+}\dd\mu_\varepsilon)}\\
&=\iint_{\Rd}\big|\nabla_v \psi \cdot\nabla_xw\big|^2\,\wv^{2\,(1-\beta)_+}\dd\widetilde\mu_\varepsilon\le\frac1{\sigma_\varepsilon^2}\,\|(1-\sfPi_\varepsilon)f\|^2_{\rmL^2\big(\wv^{-2\,(1-\beta)}\dd\mu_\varepsilon\big)}\,.
\end{align*}
Then~\eqref{eq:TAeps} follows by the Cauchy-Schwarz inequality.

\Step{}We claim that
\be{eq:Ellip}
\abs{\bangle{\sfA_\varepsilon\sfT(1-\sfPi_\varepsilon)f,\sfPi_\varepsilon\,f}_{\cH_\varepsilon}}\lesssim \norm{(1-\sfPi_\varepsilon)f}_{\rmL^2\big(\wv^{-2\,(1-\beta)_+}\dd\mu_\varepsilon\big)} \bangle{\sfA_\varepsilon\sfT\sfPi_\varepsilon\,f,\sfPi_\varepsilon\,f}_{\cH_\varepsilon}^{1/2}\,.
\ee
Let $u=u(x)$ be defined by~\eqref{eq:u}, we have $(\sfA_\varepsilon \sfT)^\ast\sfPi_\varepsilon\,f=-\,\sfT^2(u\,f_\star)$ and
\be{eq:step3}
\begin{aligned}
\bangle{\sfA_\varepsilon\sfT(1-\sfPi_\varepsilon)f, &\sfPi_\varepsilon\,f}_{\rmL^2(e^{\varepsilon E})}=\bangle{(1-\sfPi_\varepsilon)f,-\sfT^2(u\,f_\star)}_{\cH_\varepsilon}\\
&\le\norm{(1-\sfPi_\varepsilon)f}_{\rmL^2\big(\wv^{-2\,(1-\beta)_+}\dd\mu_\varepsilon\big)}\,\norm{\sfT^2 (u\,f_\star)}_{\rmL^2(\wv^{2\,(1-\beta)_+}\dd\mu_\varepsilon)}\,.
\end{aligned}
\ee
A direct computation yields
\[
\sfT^2(u\,f_\star)=\Big(\mathrm{Hess}_x(u):\nabla_v\psi\otimes\nabla_v\psi-\mathrm{Hess}_v(\psi):\nabla_xu\otimes\nabla_x\phi\Big)\,f_\star\,.
\]
Adding and subtracting $(1-\varepsilon)\,(\nabla_xu\otimes\nabla_x\phi:\nabla_v\psi\otimes\nabla_v\psi)$, we can rewrite
\begin{align*}
\sfT^2(u\,f_\star)&=\Big(\big(\mathrm{Hess}(u)-(1-\varepsilon)\,\nabla_xu\otimes\nabla_x\phi\big):\nabla_v\psi\otimes\nabla_v\psi\\
&\qquad-\big(\mathrm{Hess}(\psi)-(1-\varepsilon)\,\nabla_v\psi\otimes\nabla_v\psi\big):\nabla_xu\otimes\nabla_x\phi\Big)\,f_\star
\end{align*}
where $\nabla_xu\otimes\nabla_x\phi$ and $\nabla_v\psi\otimes\nabla_v\psi$ respectively denote the matrices with entries $\big(\partial_{x_i}u\,\partial_{x_j}\phi\big)_{i,j}$ and $\big(\partial_{v_i}\psi\,\partial_{v_j}\psi\big)_{i,j}$. We also use the notation $\mathsf a:\mathsf b:=\sum_{1\le i,j\le d}\mathsf a_{i,j}\,\mathsf b_{i,j}$ and $\|\mathsf a\|^2:=\mathsf a:\mathsf a$.
\\[4pt]
(1) Concerning the second term in the expression of $\sfT^2(u\,f_\star)$, we have
\begin{align*}
&\big\|\big(\mathrm{Hess}(\psi)-(1-\varepsilon)\,\nabla_v\psi\otimes\nabla_v\psi\big):\nabla_xu\otimes\nabla_x\phi\big\|^2_{\rmL^2(\wv^{2\,(1-\beta)_+}\dd\widetilde\mu_\varepsilon)}\\
&\le\iint_\RRd \big\|\,\mathrm{Hess}(\psi)-(1-\varepsilon)\,\nabla_v\psi\otimes\nabla_v\psi\big\|^2\,|\nabla_xu|^2\,|\nabla_x\phi|^2\,\wv^{2\,(1-\beta)_+}\dd\widetilde\mu_\varepsilon\\
&\le C_{\beta,1}(\varepsilon)\irdx{|\nabla_xu|^2\,|\nabla_x\phi|^2\dd\widetilde\nu_\varepsilon}
\end{align*}
where $C_{\beta,1}(\varepsilon)=\irdv{\big|\,\mathrm{Hess}(\psi)-(1-\varepsilon)\,\nabla_v\psi\otimes\nabla_v\psi\big|^2\,\wv^{2\,(1-\beta)_+}\,\frac{e^{-\,(1-\varepsilon)\,\psi}}{\irdv{e^{-\,(1-\varepsilon)\,\psi}}}}$ is positive and does not vanish as $\varepsilon\to0_+$. Using~\cite[Lemma 8]{MR3324910} and the Poincaré inequality~\eqref{Ineq:WeightedPoincare} we have
\be{ImprPoinc}
\irdx{|\nabla_xu|^2\,|\nabla_x\phi|^2\dd\widetilde\nu_\varepsilon}\lesssim\irdx{\rho_f^2\,\dd\nu_\varepsilon}\lesssim \bangle{\sfA_\varepsilon\sfT\sfPi_\varepsilon\,f,\sfPi_\varepsilon\,f}_{\cH_\varepsilon}
\ee
we obtain the bound
\begin{multline*}
\big\|\big(\mathrm{Hess}(\psi)-(1-\varepsilon)\,\nabla_v\psi\otimes\nabla_v\psi\big):\nabla_xu\otimes\nabla_x\phi\big\|^2_{\rmL^2(\wv^{2\,(1-\beta)_+}\dd\widetilde\mu_\varepsilon)}\\
\lesssim\,\bangle{\sfA_\varepsilon\sfT\sfPi_\varepsilon\,f,\sfPi_\varepsilon\,f}_{\cH_\varepsilon}
\end{multline*}
using~\eqref{eq:ATPi}.
\\[4pt]
(2) For the first term in the expression of $\sfT^2(u\,f_\star)$, we have
\begin{align*}
\big\|\big(\mathrm{Hess}(u)&-(1-\varepsilon)\,\nabla_xu\otimes\nabla_x\phi\big):\nabla_v\psi\otimes\nabla_v\psi\big\|^2_{\rmL^2(\wv^{2\,(1-\beta)_+}\dd\widetilde\mu_\varepsilon)}\\
&\le\iint_\RRd |\nabla_v\psi|^4\,\big\|\,\mathrm{Hess}(u)-(1-\varepsilon)\,\nabla_xu\otimes\nabla_x\phi\big\|^2\,\wv^{2\,(1-\beta)_+}\dd\widetilde\mu_\varepsilon\\
&\le C_{\beta,2}(\varepsilon)\int_\Rd\big\|\,\mathrm{Hess}(u)-(1-\varepsilon)\,\nabla_xu\otimes\nabla_x\phi\big\|^2\dd\widetilde\nu_\varepsilon
\end{align*}
where $C_{\beta,2}(\varepsilon)=\irdv{|\nabla_v\psi|^4\,\wv^{2\,(1-\beta)_+}\,e^{-\,(1-\varepsilon)\,\psi}}\big/\irdv{e^{-\,(1-\varepsilon)\,\psi}}$, which does not vanish as $\varepsilon\to0_+$. Notice that $\mathrm{Hess}(u)-(1-\varepsilon)\,\nabla_xu\otimes\nabla_x\phi$ is the matrix with entries
\[
\partial_{x_ix_j}u-(1-\varepsilon)\,\partial_{x_i}u\,\partial_{x_j}\phi=\partial_{x_i}\big(\rho_\star^{(\varepsilon)}\,\partial_{x_j}u\big)/\rho_\star^{(\varepsilon)}
\]
for $i,j=1,\dots,d$. Hence, with the same method as in~\cite{MR4769957}, namely
\begin{multline*}
\int_\Rd\(\partial_{x_i}\big(\rho_\star^{(\varepsilon)}\,\partial_{x_j}u\big)\)^2\dd\nu_\varepsilon=-\,(1-\varepsilon)\irdx{\partial_{x_i}\phi\,\partial_{x_j}u\,\partial_{x_i}\big(\rho_\star^{(\varepsilon)}\,\partial_{x_j}u\big)}\\+\irdx{\partial_{x_ix_j}^2u\,\partial_{x_i}\big(\rho_\star^{(\varepsilon)}\,\partial_{x_j}u\big)}
\end{multline*}
and an integration by parts, we have
\begin{multline*}
\int_\Rd\big\|\,\mathrm{Hess}(u)-(1-\varepsilon)\,\nabla_xu\otimes\nabla_x\phi\big\|^2\dd\widetilde\nu_\varepsilon=\sum_{i,j=1}^d\,\int_\Rd\(\partial_{x_i}\big(\rho_\star^{(\varepsilon)}\,\partial_{x_j}u\big)\)^2\dd\nu_\varepsilon\\
=\int_\Rd\big\|\nabla_x\cdot\big(\rho_\star^{(\varepsilon)}\,\nabla_xu\big)\big\|^2\dd\nu_\varepsilon+(1-\varepsilon)\int_{\R^d}|\nabla_xu|^2\,\Delta_x\phi\dd\widetilde\nu_\varepsilon\\
-(1-\varepsilon)\int_{\R^d}\mathrm{Hess}(\phi):\nabla_xu\otimes\nabla_xu\dd\widetilde\nu_\varepsilon\,.
\end{multline*}
All integrals on the right-hand side are bounded up to constants by $\bangle{\sfA_\varepsilon\sfT\sfPi_\varepsilon\,f,\sfPi_\varepsilon\,f}_{\cH_\varepsilon}$ using the improved Poincaré inequality~\eqref{ImprPoinc} and
\[
\big\|\big(\mathrm{Hess}(u)-(1-\varepsilon)\,\nabla_xu\otimes\nabla_x\phi\big):\nabla_v\psi\otimes\nabla_v\psi\big\|^2_{\rmL^2(\wv^{2\,(1-\beta)_+}\dd\widetilde\mu_\varepsilon)}\lesssim \bangle{\sfA_\varepsilon\sfT\sfPi_\varepsilon\,f,\sfPi_\varepsilon\,f}_{\cH_\varepsilon}
\]
which follows from~\eqref{eq:ATPi}. Taking~\eqref{eq:step3} into account, we conclude that
\[
\|\sfT^2(u\,f_\star)\|_{\rmL^2(\wv^{2\,(1-\beta)_+}\dd\mu_\varepsilon)}^2\lesssim\bangle{\sfA_\varepsilon\sfT\sfPi_\varepsilon\,f,\sfPi_\varepsilon\,f}_{\dd\mu_\varepsilon)}\,.
\]
This completes the proof of~\eqref{eq:Ellip}.

\Step{}We claim that
\be{ALf}
\abs{\bangle{\sfA_\varepsilon\sfL f,f}_{\cH_\varepsilon}}\lesssim \norm{(1-\sfPi_\varepsilon)f}_{\rmL^2\big(\wv^{-2\,(1-\beta)_+}\dd\mu_\varepsilon\big)} \bangle{\sfA_\varepsilon\sfT\sfPi_\varepsilon\,f,\sfPi_\varepsilon\,f}_{\cH_\varepsilon}^{1/2}\,.
\ee
Notice that $\sfL\sfPi_\varepsilon=0$, so we can write
\begin{align*}
\bangle{\sfA_\varepsilon\sfL f,f}_{\cH_\varepsilon}=\bangle{\sfA_\varepsilon \sfL(1-\sfPi_\varepsilon)f, \sfPi_\varepsilon\,f}_{\cH_\varepsilon}=\bangle{(1-\sfPi_\varepsilon)f, \sfL^*\sfA_\varepsilon^* \sfPi_\varepsilon\,f}_{\cH_\varepsilon}\,.
\end{align*}
By the Cauchy-Schwarz inequality we have
\[
\abs{\bangle{\sfA_\varepsilon\sfL f,f}_{\cH_\varepsilon}}\le\norm{(1-\sfPi_\varepsilon)f}_{\rmL^2\big(\wv^{-2\,(1-\beta)_+}\dd\mu_\varepsilon\big)}\,\norm{\sfL^*\sfA_\varepsilon^*\sfPi_\varepsilon\,f}_{\rmL^2(\wv^{2\,(1-\beta)_+}\dd\mu_\varepsilon)}\,.
\]
Let $u$ defined by~\eqref{eq:u} and notice that $\sfL^*\sfA_\varepsilon^*\sfPi_\varepsilon\,f=\sfL^*\sfT(u\,f_\star)=\sfL^*(\nabla_v \psi\cdot\nabla_xu\,f_\star)$. Recalling~\eqref{eq:L*} we have
\begin{align*}
\sfL^*\sfA_\varepsilon^*\sfPi_\varepsilon\,f &=\sfL(\nabla_v \psi\cdot\nabla_xu\,f_\star) + 2\,\varepsilon\,\nabla_v \psi \cdot\nabla_v (\nabla_v \psi\cdot\nabla_xu)\,f_\star\\
&\qquad\qquad\qquad+\,\varepsilon\,\nabla_v \psi\cdot\nabla_xu\,\big(\Delta_v \psi -(1-\varepsilon)\,\abs{\nabla_v \psi}^2\big)\,f_\star\\
&=\nabla_xu\cdot\Big(\nabla_v \cdot\mathrm{Hess}_v(\psi)- (1-2\,\varepsilon) \nabla_v \psi \cdot\mathrm{Hess}_v(\psi)\\
&\qquad\qquad\qquad+\,\varepsilon\,\big(\Delta_v \psi -(1-\varepsilon)\,\abs{\nabla_v \psi}^2\big)\,\nabla_v \psi\Big)\,f_\star\\
&=:\big(\nabla_xu\cdot\xi_\varepsilon(v)\big)\,f_\star
\end{align*}
where $\xi_\varepsilon$ is a vector valued function in $v$ with power law growth. Hence
\begin{align*}
\norm{\sfL^*\sfA_\varepsilon^*\sfPi_\varepsilon\,f}^2_{\rmL^2(\wv^{2\,(1-\beta)_+}\dd\mu_\varepsilon)}&\le\int_\Rd\abs{\nabla_xu}^2\(\int_\Rd\abs{\xi_\varepsilon(v)}^2\,\wv^{2\,(1-\beta)_+}\,e^{\varepsilon E}\,f_\star\dv \) \dx\\
&\le C_\beta(\varepsilon)\,\norm{\nabla_xu}^2_{\rmL^2(\dd\widetilde\nu_\varepsilon)}
\end{align*}
where
\[
C_\beta(\varepsilon)=\irdv{\frac{|\xi_\varepsilon(v)|^2\,\wv^{2\,(1-\beta)_+}\,e^{-\,(1-\varepsilon)\,\psi}}{\irdv{e^{-\,(1-\varepsilon)\,\psi}}}}\,.
\]
Notice that, as $\varepsilon\to0_+$, the constant $C_\beta(\varepsilon)$ does not vanish as $\varepsilon\to0_+$. This completes the proof of~\eqref{ALf}.

\Step{}Let us show that
\be{AL}
\abs{\bangle{\sfA_\varepsilon\,f,\sfL f}_{\cH_\varepsilon}}\lesssim \varepsilon\,\norm{(1-\sfPi_\varepsilon)f}^2_{\rmL^2\big(\wv^{-2\,(1-\beta)_+}\dd\mu_\varepsilon\big)}\,.
\ee
Using $\sfPi_\varepsilon \sfA_\varepsilon=\sfA_\varepsilon$, we have by the Cauchy-Schwarz inequality
\begin{align*}
\abs{\bangle{\sfA_\varepsilon\,f,\sfL f}_{\cH_\varepsilon}} &\le\abs{\bangle{\sfA_\varepsilon\,f,\sfPi_\varepsilon\sfL f}_{\cH_\varepsilon}}\\
&\le\norm{\sfA_\varepsilon\,f}_{\rmL^2\big(\wv^{-2\,(1-\beta)_+}\dd\mu_\varepsilon\big)}\,\norm{\sfPi_\varepsilon\sfL f}_{\rmL^2(\wv^{2\,(1-\beta)_+}\dd\mu_\varepsilon)}\,.
\end{align*}
$\bullet$ For the first norm of the right-hand side, let us consider $w=w(x)$ defined as in~\eqref{eq:w}. By reasoning as in Step~2, we have
\begin{align*}
\int_{\R^d}|w|^2\dd\widetilde\nu_\varepsilon&+\sigma_\varepsilon\int_{\R^d}|{\nabla_xw}|^2\dd\widetilde\nu_\varepsilon=\irdx{\nabla_xw\cdot\(\irdv{\nabla_v\psi\,(1-\sfPi_\varepsilon)f\,e^{\varepsilon E}}\)}\\
&\le\sigma_\varepsilon\int_{\R^d}|{\nabla_xw}|^2\dd\widetilde\nu_\varepsilon + \frac1{4\,\sigma_\varepsilon}\,\norm{(1-\sfPi_\varepsilon)f}^2_{\rmL^2\big(\wv^{-2\,(1-\beta)_+}\dd\mu_\varepsilon\big)}\,.
\end{align*}
In particular
\[
\int_{\R^d}|w|^2\dd\widetilde\nu_\varepsilon\le\frac1{4\,\sigma_\varepsilon}\,\norm{(1-\sfPi_\varepsilon)f}^2_{\rmL^2\big(\wv^{-2\,(1-\beta)_+}\dd\mu_\varepsilon\big)}\,,
\]
which gives
\begin{align*}
\|\sfA_\varepsilon\,f\,\|^2_{\rmL^2(\wv^{2\,(1-\beta)_+}\dd\mu_\varepsilon)}&=\|w\,f_\star\|^2_{\rmL^2(\wv^{2\,(1-\beta)_+}\dd\mu_\varepsilon)}=\irdxvfstar{|w|^2\,\wv^{2\,(1-\beta)_+}\,e^{\varepsilon E}}\\
&\lesssim \|(1-\sfPi_\varepsilon)f\|^2_{\rmL^2\big(\wv^{-2\,(1-\beta)}\dd\mu_\varepsilon\big)}\,.
\end{align*}
$\bullet$ Concerning the second norm, by a direct computation we obtain
\begin{align*}
\sfPi_\varepsilon \sfL f &=\frac{f_\star}{\rho_\star^{(\varepsilon)}}\int_\Rd\nabla_v \cdot\big(f_\star\,\nabla_v (f/f_\star)\big)\,e^{\varepsilon E}\dv\\&=\frac{f_\star}{\rho_\star^{(\varepsilon)}}\int_\Rd\nabla_v \cdot\Big(f_\star\,\nabla_v \big((1-\sfPi_\varepsilon)f\,f_\star^{-1}\big)\Big)\,e^{\varepsilon E}\dv\\
&=\frac{f_\star}{\rho_\star^{(\varepsilon)}}\int_\Rd(1-\sfPi_\varepsilon)f\,f_\star^{-1}\,\nabla_v \cdot\big(f_\star\nabla_v\,e^{\varepsilon E}\big)\dv\\
&=\varepsilon\,\frac{f_\star}{\rho_\star^{(\varepsilon)}}\int_\Rd(1-\sfPi_\varepsilon)f\,\big(\Delta_v\psi - (1-\varepsilon)\,\abs{\nabla_v \psi}^2\big)\,e^{\varepsilon E}\dv\,,
\end{align*}
so, using $\big|\Delta_v\psi - (1-\varepsilon)\,\abs{\nabla_v \psi}^2\big|\lesssim \wv^{2\,(\beta-1)}$,
\begin{align*}
\abs{\sfPi_\varepsilon \sfL f}^2\lesssim\,&\,\varepsilon^2\,\frac{f_\star^2}{(\rho_\star^{(\varepsilon)})^2} \(\int_\Rd\abs{(1-\sfPi_\varepsilon)f}^2\,\wv^{-2\,(1-\beta)_+}\,f_\star^{-1}\dv\) \(\int_\Rd f_\star\,e^{\varepsilon E}\dv \)\\
&=\varepsilon^2\,\frac{f_\star^2}{\rho_\star^{(\varepsilon)}} \int_\Rd\abs{(1-\sfPi_\varepsilon)f}^2\,\wv^{-2\,(1-\beta)_+}\,f_\star^{-1}\dv\,.
\end{align*}
As a result, we obtain
\begin{align*}
\norm{\sfPi_\varepsilon\sfL f}_{\rmL^2(\wv^{2\,(1-\beta)_+}\dd\mu_\varepsilon)}^2\lesssim \varepsilon^2\,\norm{(1-\sfPi_\varepsilon)f}^2_{\rmL^2\big(\wv^{-2\,(1-\beta)_+}\dd\mu_\varepsilon\big)}\,.
\end{align*}
This completes the proof of~\eqref{AL}.

\Step{}Finally, by collecting all estimates, we have

\[
\sfD_\varepsilon[f]\gtrsim Q_\delta(X,Y)-\,\varepsilon\,\norm f^2_{\rmL^2\big(\wv^{-2\,(1-\beta)_+}\dd\mu_\varepsilon\big)}
\]
where $X=\norm{(1-\sfPi_\varepsilon)f}_{\rmL^2\big(\wv^{-2\,(1-\beta)_+}\dd\mu_\varepsilon\big)}$, $Y=\bangle{\sfA_\varepsilon\sfT\sfPi_\varepsilon\,f,\sfPi_\varepsilon\,f}_{\cH_\varepsilon}^{1/2}$ and
\[
Q_\delta(X,Y):=(1-\delta)\,X^2-\delta\,X\,Y+\delta\,Y^2
\]
is a positive quadratic form under the discriminant condition $\delta^2-4\,\delta\,(1-\delta)<0$, \emph{i.e.}, $0<\delta<4/5$. Taking $\delta$ in this range, small enough, completes the proof.
\end{proof}

In order to study the convergence of a solution of~\eqref{kFP} with nonnegative initial datum $f_0$ such that $\nrmxv{f_0}1=1$ to $f_\star\in\cH_\varepsilon$ as $t\to+\infty$, we consider the space $\overline\cH_\varepsilon:=\big\{f\in\cH_\varepsilon\,:\,\irdxv f=0\big\}$ of function with zero mass and study the convergence of \hbox{$f-f_\star\in\overline\cH_\varepsilon$ to~$0$} as $t\to+\infty$.
\begin{lemma}\label{lem:1}
For any $\beta>0$ and $\alpha\ge1$, there is a positive constant $\kappa$ such that
\be{eq:D_eps_2}
\sfD_\varepsilon[f]+2\,\varepsilon\,\norm f^2_{\rmL^2\big(\wv^{-2\,(1-\beta)_+}\dd\mu_\varepsilon\big)}\ge\kappa \(\|(1-\sfPi_\varepsilon)f\|_{\rmL^2\big(\wv^{-2\,(1-\beta)_+}\dd\mu_\varepsilon\big)}^2 + \|\sfPi_\varepsilon\,f\|_{\cH_\varepsilon}^2\)
\ee
for any $f\in\overline\cH_\varepsilon$. In particular, if $\varepsilon>0$ is small enough, we have
\be{eq:D_eps_2b}
\sfD_\varepsilon[f]\gtrsim \norm f^2_{\rmL^2\big(\wv^{-2\,(1-\beta)_+}\dd\mu_\varepsilon\big)}\ge0\quad\forall\,f\in\overline\cH_\varepsilon\,.
\ee
\end{lemma}
\begin{proof} We divide the proof in three steps.

\setcounter{Step}0\Step{} Let $u=u(x)$ be such that $u\,f_\star=\big(1+ (\sfT\sfPi_\varepsilon)^*(\sfT\sfPi_\varepsilon)\big)^{-1}\sfPi_\varepsilon\,f$. By squaring~\eqref{eq:u} and integrating in $x$ with respect to $\dd\widetilde\nu_\varepsilon$, we have
\begin{align*}
\int_\Rd u^2\dd\widetilde\nu_\varepsilon + 2\,\sigma_\varepsilon\int_\Rd\abs{\nabla_xu}^2 \dd\widetilde\nu_\varepsilon + \sigma_\varepsilon^2\int_\Rd\abs{\nabla_x \cdot(\rho_\star^{(\varepsilon)}\,\nabla_xu)}^2\dd\nu_\varepsilon=\int_\Rd\big|\rho_f^{(\varepsilon)}\big|^2\dd\nu_\varepsilon\,.
\end{align*}
Using~\eqref{eq:ATPi}, we obtain
\be{eq:1140}
\|\sfPi_\varepsilon\,f\|_{\cH_\varepsilon}^2=\int_\Rd\big|\rho_f^{(\varepsilon)}\big|^2\dd\nu_\varepsilon\le\int_\Rd u^2\dd\widetilde\nu_\varepsilon + 2\,\bangle{\sfA_\varepsilon\sfT\sfPi_\varepsilon\,f,\sfPi_\varepsilon\,f}_{\cH_\varepsilon}\,.
\ee
By the Poincaré inequality~\eqref{Ineq:WeightedPoincare} applied with $\overline u_\varepsilon=\int_\Rd u\tildemeas\alpha$ to
\[
\dd\widetilde\nu_\varepsilon=\(\int_\Rd\dd\widetilde\nu_\varepsilon\)\tildemeas\alpha=\(\irdx{\rho_\star^{(\varepsilon)}}\)\tildemeas\alpha\,,
\]
we obtain
\begin{align*}
\int_\Rd u^2\dd\widetilde\nu_\varepsilon&=\int_\Rd\abs{u-\overline u_\varepsilon}^2 \dd\widetilde\nu_\varepsilon + \int_\Rd\overline u_\varepsilon^2\dd\widetilde\nu_\varepsilon\\
&\le\mathscr C_{\gamma,\varepsilon}\int_\Rd\abs{\nabla_xu}^2 \dd\widetilde\nu_\varepsilon + \overline u_\varepsilon^2 \int_\Rd\dd\widetilde\nu_\varepsilon\lesssim \bangle{\sfA_\varepsilon\sfT\sfPi_\varepsilon\,f,\sfPi_\varepsilon\,f}_{\cH_\varepsilon} + \(\int_\Rd\rho_f^{(\varepsilon)}\dx \)^2
\end{align*}
for $\varepsilon>0$ small enough, after taking advantage of $\lim_{\varepsilon\to0_+}\int_\Rd\dd\widetilde\nu_\varepsilon=\int_\Rd\rho_\star\dx=1$. Inserting this estimate into~\eqref{eq:1140} gives
\[
\|\sfPi_\varepsilon\,f\|_{\cH_\varepsilon}^2\lesssim \bangle{\sfA_\varepsilon\sfT\sfPi_\varepsilon\,f,\sfPi_\varepsilon\,f}_{\cH_\varepsilon} + \(\int_\Rd\rho_f^{(\varepsilon)}\dx \)^2\,.
\]

\Step{}Since $f\in\overline\cH_\varepsilon$, we have $\lim_{\varepsilon\to0_+}\int_\Rd\rho_f^{(\varepsilon)}\dx=\int_\Rd\rho_f\dx=0$ and
\begin{align*}
\abs*{\int_\Rd\rho_f^{(\varepsilon)}\dx} &=\abs*{\int_\Rd\rho_f^{(\varepsilon)}\dx- \int_\Rd\rho_f\dx}=\abs*{ \int_0^\varepsilon\,\frac{\mathrm d}{\mathrm d\eta}\(\iint_\RRd f\,e^{\eta E}\dx\dv\)\dd \eta}\\
&=\abs*{ \int_0^\varepsilon \(\iint_\RRd f\,E\,e^{\eta E}\dx\dv\)\dd \eta}\\
&\le\int_0^\varepsilon\,\norm f_{\rmL^2\big(\wv^{-2\,(1-\beta)_+}\,e^{\eta E}\dd\mu_\star\big)}\,\norm{E}_{\rmL^2\big(\wv^{2\,(1-\beta)_+}\,e^{\eta E}f_\star\big)} \dd \eta\\
&\lesssim \varepsilon\,\norm f_{\rmL^2\big(\wv^{-2\,(1-\beta)_+}\dd\mu_\varepsilon\big)}\,.
\end{align*}
Hence we conclude that
\[
\|\sfPi_\varepsilon\,f\|_{\cH_\varepsilon}^2\lesssim \bangle{\sfA_\varepsilon\sfT\sfPi_\varepsilon\,f,\sfPi_\varepsilon\,f}_{\cH_\varepsilon} +\varepsilon^2\,\norm f^2_{\rmL^2\big(\wv^{-2\,(1-\beta)_+}\dd\mu_\varepsilon\big)}\,.
\]
With Lemma~\ref{lem:Dissipation1}, this completes the proof of~\eqref{eq:D_eps_2}.

\Step{} Finally we remark that
\begin{align*}
\norm f^2_{\rmL^2\big(\wv^{-2\,(1-\beta)_+}\dd\mu_\varepsilon\big)} &\lesssim \norm{(1-\sfPi_\varepsilon)f}^2_{\rmL^2\big(\wv^{-2\,(1-\beta)_+}\dd\mu_\varepsilon\big)} + \norm{\sfPi_\varepsilon\,f}^2_{\rmL^2\big(\wv^{-2\,(1-\beta)_+}\dd\mu_\varepsilon\big)}\\
&\qquad\qquad\le\norm{(1-\sfPi_\varepsilon)f}^2_{\rmL^2\big(\wv^{-2\,(1-\beta)_+}\dd\mu_\varepsilon\big)} + \norm{\sfPi_\varepsilon\,f}^2_{\cH_\varepsilon}\,.
\end{align*}
Thus, if $\varepsilon>0$ is small enough, we have
\[
\sfD_\varepsilon[f]\gtrsim \norm f^2_{\rmL^2\big(\wv^{-2\,(1-\beta)_+}\dd\mu_\varepsilon\big)}\,,
\]
which completes the proof of~\eqref{eq:D_eps_2b}.
\end{proof}

\begin{proof}[Proof of Proposition~\ref{prop:energyMomDMS}, Estimate~\eqref{eq:expBounded}]
We learn from~\eqref{eq:D_eps_2b} and~\eqref{Df2} that $t\mapsto\sfH_\varepsilon[f(t,\cdot,\cdot)]$ is non-increasing, which proves~\eqref{eq:expBounded} using Lemma~\ref{Lem:Equivalence-eps}.
\end{proof}

\subsubsection{Operator interpolation}\label{Sec:interp}

The generalization of the Stein-Weiss theorem~\cite[Theorem 2.9]{MR92943} proposed by Bergh and L\"ofstr\"om in~\cite[Section 5]{BL76} applies. A positive function $h\colon\R_+\to\R_+$ is \emph{quasi-concave} if $h(s) \asymp k(s)$ for some concave function $k\colon\R_+\to\R_+$. A \emph{quasi-norm} is defined by the same properties as a norm, except that the triangle inequality is satisfied up only to a positive constant. Let $w_0$, $w_1$, $\widetilde w_0$ and $\widetilde w_1$ be positive weight functions. According to~\cite[Definition~5.4.2]{BL76}, a positive function $h$ is called an \emph{interpolation function of power} $p\in(0,+\infty]$ if, for an operator~$T$ such that $T\colon \rmL^p(w_0\dx\dv)\to \rmL^p(\widetilde w_0\dx\dv)$ and $T\colon \rmL^p(w_1\dx\dv)\to \rmL^p(\widetilde w_1\dx\dv)$ with quasi-norms $M_0>0$ and $M_1>0$, where $w_0$, $w_1$, $\widetilde w_0$ and $\widetilde w_1$ are positive weight functions, then
\[
T\colon \rmL^p\Big(w_0\,h\big(\tfrac{w_1}{w_0}\big)\dx\dv\Big)\to \rmL^p\Big(\widetilde w_0\,h\big(\tfrac{\widetilde w_1}{\widetilde w_0}\big)\dx\dv\Big)
\]
with quasi-norm $M\lesssim\max\{M_0,M_1\}$.
\begin{theorem}{\rm \cite[Theorem~5.4.4]{BL76}}\label{Thm:Stein-Weiss-variant} A positive function $h$ is an interpolation function of power $p$ if and only if it is quasi-concave. In particular, if $h$ is an interpolation function of power $p$ for some $p$, the same is true for all $p$.\end{theorem}
\begin{proof}[Proof of Proposition~\ref{prop:energyMomDMS}, Estimate~\eqref{eq:polyBounded}]
We apply this result with $p=2$, $T=S_\cL(t)$,
\[
w_0=\widetilde w_0=f_\star^{-1}\quad\mbox{and}\quad w_1=\widetilde w_1 = e^{\varepsilon E}\,f_\star^{-1}
\]
and a concave function $h$ such that $h(s)\asymp(\ln s)^k$ for $s$ large. According to~\eqref{eq:expBounded}, for some $\varepsilon>0$ fixed small enough,
\[
\norm{S_\cL(t)}_{\rmL^2(\dd\mu_\star)\to \rmL^2(\dd\mu_\star)}\le 1\quad\mbox{and}\quad \norm{S_\cL(t)}_{\cH_\varepsilon\to \cH_\varepsilon}\lesssim 1\,.
\]
Since $\widetilde w_0\,h\big(\widetilde w_1/\widetilde w_0\big)=w_0\,h\big(w_1/w_0\big)\asymp f_\star^{-1} \big(\ln(e^{\varepsilon E})\big)^k \asymp E^k\,f_\star^{-1}$, we can conclude that
\[
\norm{S_\cL(t)}_{\rmL^2\(E^k\dd\mu_\star\)\to \rmL^2\(E^k\dd\mu_\star\)}\lesssim 1\,.
\]
This estimate establishes~\eqref{eq:polyBounded} and completes the proof of Proposition~\ref{prop:energyMomDMS}.
\end{proof}

\section{\texorpdfstring{$\rmL^2$}{L2} moment bounds based on Lyapunov functions}\label{Sec:Lyapunov}

Here we establish weighted $\rmL^2$ estimates using weights satisfying a Lyapunov condition, in order to cover the ranges $\alpha>0$ and $\beta>1$ on the one hand Section~\ref{Sec:Lyapunov1}), and $\max\{\alpha>0,\beta\}<1$ on the other hand (Section~\ref{Sec:Lyapunov}).

\subsection{An abstract result}\label{Sec:Moments}

We define the \emph{dual} of~$\cL$ with respect to the standard $\rmL^2(\dd x\dv)$ norm, or \emph{formal adjoint}, as
\be{L}
\cL^\star m:=\nabla_v \psi\cdot\nabla_x m-\nabla_x \phi\cdot \nabla_v m +\Delta_v m-\nabla_v \psi\cdot\nabla_v m
\ee
using integrations by parts, in the sense that
\[
\irdxv{(\cL f)\,m}=\irdxv{f\,\cL^\star m}
\]
for compactly supported and smooth enough test functions $f$. We shall say that $m$ is a function satisfying the \emph{Lyapunov condition}, or a \emph{Lyapunov function}, if $m$ is continuous function $m$ on $\RRd$ taking positive values such that, for some positive constants $C$, $R$, $\eta$, and $\ell$, we have
\[
\cL^\star m\le C\,\one_{B_R}-\eta\,m^{1-\ell}\,.
\]
\begin{proposition}\label{Prop:Lyapunov} Let $k>0$. Assume that $f$ solves~\eqref{kFP} with nonnegative initial datum $f_0\in\rmL^1(\dx\dv)$ such that \hbox{$\nrmxv{f_0}1=1$}. Assume that $m^k$ is a \emph{Lyapunov function} such that
\be{Eqn:Lyapunov}
\cL^\star m^k\le C\,\one_{B_R}-\eta\,m^{k-\ell}\,,
\ee
holds for some $\ell\in(0,k)$ and positive constants $C$, $R$ and $\eta$, where $B_R$ is the centred ball of radius $R$ in $\RRd$. If $f_0\,\big(1+m^k\big)\in\rmL^1(\dx\dv)$, then we have
\[
\irdxv{f(t,x,v)\,m^k(x,v)}\le\irdxv{f_0\,m^k}+2\,C\,m_R^{k+\ell}\,\eta^{-1}\quad\forall\,t\ge0\,,
\]
where $m_R:=\sup_{B_R}m$.
\end{proposition}
\begin{proof} Let us consider on $\rmL^1(\dx\dv)$ the operators
\[
\cA\,f:=C\,\one_{B_R}\,f\quad\mbox{and}\quad\cB\,f:=-\,\eta\,m^{k-\ell}\,f
\]
such that
\begin{multline*}
\frac{\mathrm d}{\mathrm dt}\irdxv{f\,m^k}=\irdxv{f\,\cL^\star (m^k)}\\
\le\irdxv{\cA\,f}+\irdxv{\cB\,f}
\end{multline*}
if $f$ solves~\eqref{kFP} with initial condition $f_0$. For any $t\ge0$, we have
\[
\irdxv{f(t,x,v)\,m^k(x,v)}\le y(t)
\]
where $t\mapsto y(t)$ is such that
\begin{align*}
&y'(t)=\irdxv{\cA\,f}+\irdxv{\cB\,f}\,,\\
&y(0)=y_0:=\irdxv{f_0\,m^k}\,.
\end{align*}

Let $\mu>0$ be a real parameter and consider $t\mapsto z(t;\mu,g_0)$ given by
\[
z'(t)=-\,\eta\irdxv{m^{\mu-\ell}\,g}\,,\quad z(0)=\irdxv{g_0\,m^\mu}\,,
\]
where $g$ solves
\be{BG}
\partial_tg=\cB\,g\quad\mbox{with}\quad g(t=0,\cdot,\cdot)=g_0\,.
\ee
The Duhamel formula shows that
\[
y(t)=z(t;k,f_0)+\int_0^tz\big(t-s;k,(\cA\,f)\,(s,\cdot,\cdot)\big)\dd s\,.
\]
Let $k'\ge k$. Using $z'<0$ from~\eqref{BG}, we learn that
\[
z(t;k',g_0)\le z(0;k',g_0)=\irdxv{g_0\,m^{k'}}\quad\forall\,t\ge0
\]
and, as a consequence if $k'=k$, we deduce that
\[
y(t)\le y_0+\int_0^tz\big(t-s;k,(\cA\,f)\,(s,\cdot,\cdot)\big)\dd s\,.
\]
Using H\"older's inequality $\norm g_{\rmL^1\(m^k\)}\le\norm g_{\rmL^1\(m^{k'}\)}^{1/(1+p)}\,\norm g_{\rmL^1\(m^{k-\ell}\)}^{p/(1+p)}$ with $p=\ell/(k'-k)$, that is,
\[
\norm g_{\rmL^1\(m^{k-\ell}\)}\ge\norm g_{\rmL^1\(m^k\)}^{1+p}\,\norm g_{\rmL^1\(m^{k'}\)}^{-p}
\]
we also learn that $z=z(t;k,g_0)$ satisfies
\begin{align*}
z'\le-\,\,\eta\,\irdxv{m^{k-\ell}\,g}=&\,-\,\eta\,\norm g_{\rmL^1\(m^{k-\ell}\)}\\
\le&\,-\,\eta\,\frac{z^{1+p}}{z(t;k',g_0)^p}\le-\,\eta\,\frac{z^{1+p}}{z(0;k',g_0)^p}
\end{align*}
for any $t\ge0$. By the Bihari-LaSalle estimate, we deduce that
\[
z(t;k,g_0)\le\(\frac1{z_0^p}+\frac{p\,\eta\,t}{z(t;k',g_0)^p}\)^{-1/p}\quad\forall\,t\ge0\,,
\]
with $z_0=\irdxv{g_0\,m^k}$ and $z(t;k',g_0)\le\irdxv{g_0\,m^{k'}}=:\widetilde z_0$. Equality holds at $t=0$, \emph{i.e.}, $z(0;k,g_0)=z_0$. We choose $k'=k+2\,\ell$ so that $p=1/2$. Altogether, we have
\[
z(t;k,g_0)\le\(\frac 1{\sqrt{z_0}}+\frac\eta{2\,\sqrt{\widetilde z_0}}\,t\)^{-2}\quad\forall\,t\ge0\,.
\]
Let us consider the special case $g_0=\cA f(s,\cdot,\cdot)$. Since $g_0$ is supported in $B_R$ and $\nrm{f(s,\cdot,\cdot)}{\rmL^1}=1$, we have
\be{z0z1}
\begin{aligned}
&z_0=z(0;k,g_0)=\irdxv{\big(\cA f(s,\cdot,\cdot)\big)\,m^k}\le C\,m_R^k\,,\\
&\widetilde z_0=z(0;k+2\,\ell,g_0)=\irdxv{\big(\cA f(s,\cdot,\cdot)\big)\,m^{k+2\,\ell}}\le C\,m_R^{k+2\,\ell}\,.\\
\end{aligned}
\ee
In that case, we obtain
\[
\int_0^tz\big(t-s;k,(\cA\,f)\,(s,\cdot,\cdot)\big)\dd s\le\frac{2\,z_0\,\sqrt{\widetilde z_0}\,t}{2\,\sqrt{\widetilde z_0}+\eta\,\sqrt{z_0}\,t}\le2\,\frac{\sqrt{z_0\,\widetilde z_0}}\eta\quad\forall\,t\ge0\,,
\]
which completes the proof using~\eqref{z0z1}.
\end{proof}

\begin{corollary}\label{Cor:Stein} Under the same assumptions as in Proposition~\ref{Prop:Lyapunov}, if $m\ge1$ and $f$ solves~\eqref{kFP} with $f_0\,\in\rmL^2(m^k\,\dd\mu_\star)$, then
\[
\irdmustar{|f(t,\cdot,\cdot)|^2\,m^k}\le\big(1+2\,C\,m_R^{k+\ell}\,\eta^{-1}\big)\irdmustar{|f_0|^2\,m^k}\quad\forall\,t\ge0\,.
\]
\end{corollary}
\begin{proof} We can deduce from Proposition~\ref{Prop:Lyapunov} and the assumption $m\ge1$ that
\[
\norm{S_\cL(t)}_{\rmL^1\(m^k\dd x\dd v\)\,\to\,\rmL^1\(m^k\dd x\dd v\)}\le \cC_1:=1+2\,C\,m_R^{k+\ell}\,\eta^{-1}\,.
\]
Thanks to the maximum principle we have
\[
\norm{S_\cL(t)}_{\rmL^\infty\(f_\star^{-1}\dd x\dd v\)\,\to\,\rmL^\infty\(f_\star^{-1}\dd x\dd v\)}\le1\,.
\]
By Stein's interpolation~\cite[Theorem~2]{MR82586} or~\cite[Theorem 2.9]{MR92943}, we deduce
\[
\norm{S_\cL(t)}_{\rmL^2\(m^kf_\star^{-1}\dd x\dd v\)\,\to\,\rmL^2\(m^kf_\star^{-1}\dd x\dd v\)}^2\le\cC_1\quad\forall\,t\ge0\,.
\]
\end{proof}

\subsection{A Lyapunov function when \texorpdfstring{$\alpha>0$}{alpha>0} and \texorpdfstring{$\beta>1$}{beta>1}}\label{Sec:Lyapunov1}~

The aim of this section is to build a Lyapunov function $m$ satisfying~\eqref{Eqn:Lyapunov} which dominates both $\wx^2$ and $\wv^2$. The abstract result of Section~\ref{Sec:Moments} can then be used.
\begin{lemma}\label{lem:Lapunov}
Consider the operator $\cL^\ast$ defined in~\eqref{L} with $\beta>1$ and $\alpha>0$. Let $m\colon \RRd\to \R$ be the weight function defined by
\[
m=E^{\gamma}+\varepsilon\(x\cdot v+\tfrac12\,\wx^2\)
\]
where $\gamma=\max\{1,2/\beta\}$. If $\varepsilon>0$ is chosen sufficiently small, then $m$ is a \emph{Lyapunov function} and for any $k> 0$ there are positive constants $C$, $R$ and $\eta$ such that~\eqref{Eqn:Lyapunov} holds with $\ell=1-\min\left\{1,\alpha/2,\beta/2\right\}$.
\end{lemma}
\begin{proof} Lemma~\ref{lem:Lapunov} relies on elementary computations.

\step1{} If $\varepsilon>0$ is small enough, we claim that the function $m$ is positive. Indeed, since $\gamma\ge2/\beta$, we have
\[
\abs{x\cdot v}\lesssim\wv^2+\wx^2\lesssim E^\gamma+\tfrac12\,\wx^2\,.
\]
Moreover, since $\gamma\ge 1$, we have $E^\gamma \asymp \wv^{\gamma\,\beta} + \wx^{\gamma\,\alpha}$ and $m$ has a power law growth:
\be{equiv1}
m\asymp \wv^{\gamma\,\beta}+\wx^{\max\{\gamma\,\alpha, 2\}}\quad\mbox{as}\quad|(x,v)|\to+\infty\,.
\ee

\step2{The function $m$ is a Lyapunov function (case $k=1$).} Using $\sfT E=0$, we have
\begin{align*}
\cL^*m&=\gamma \(d+(\beta-2)\,\frac{\abs{v}^2}{\wv^2} + (\gamma-1)\,\frac{\wv^{\beta-2}\,\abs{v}^2}{E}-\wv^{\beta-2}\,\abs{v}^2\)E^{\gamma-1}\,\wv^{\beta-2}\\
&\qquad+\varepsilon\,\wv^{\beta-2}\,\abs{v}^2-\varepsilon\,\wx^{\alpha-2}\,\abs{x}^2\\
&\le\Big(C\,\one_{\abs{v}\le R}-\tfrac12\,\wv^\beta\Big)\,E^{\gamma-1}\,\wv^{\beta-2}+\varepsilon\,\wv^{\beta-2}\,\abs{v}^2-\varepsilon\,\wx^{\alpha-2}\,\abs{x}^2
\end{align*}
for some positive $C$ and $R$ large enough. Now, if $\varepsilon>0$ is small enough, we have
\begin{align*}
&\varepsilon\,\wv^{\beta-2}\,\abs{v}^2\le\varepsilon\,\wv^{2-\beta}\,\wv^{2\,(\beta-1)}\le\frac14\,E^{\gamma-1}\,\wv^{2\,(\beta-1)}\,,\\
&\cL^\star m\le C\,\one_{\abs{v}\le R}\,E^{\gamma-1}\,\wv^{\beta-2} -\frac14\,E^{\gamma-1}\,\wv^{2\,(\beta-1)} -\varepsilon\,\wx^{\alpha-2}\,\abs{x}^2\,.
\end{align*}
Next, we notice that
\[
\one_{\abs{v}\le R}\,E^{\gamma-1}\,\wv^{\beta-2}\lesssim \one_{\abs{v}\le R}\,\wx^{\alpha\,(\gamma-1)}\,.
\]
Therefore, since $\beta>1$, we have $\gamma<2$ and
\[
C\,\one_{\abs{v}\le R}\,E^{\gamma-1}\,\wv^{\beta-2} - \frac\varepsilon 2\,\wx^{\alpha-2}\,\abs{x}^2\lesssim \one_{\abs{v}\le R}\,\one_{\abs{x}\le R}\,,
\]
up to enlargements of $R$. We can conclude that
\[
\cL^*m\le C\,\one_{B_R}-\frac\varepsilon 2 \(\wv^{\gamma\,\beta+\beta-2}+\wx^\alpha\),
\]
up to enlargements of $C$.

\step3{The function $m^k$ is a Lyapunov function.} With $k>0$, we have
\begin{align*}
&\cL^*(m^k)=k\,m^{k-1} \(\cL^*m+\frac{k-1}m\,\abs{\nabla_v m}^2\)\\
&\le k\,m^{k-1}\(C\,\one_{B_R}- \frac\varepsilon 2\(\wv^{\gamma\,\beta+\beta-2}+\wx^\alpha\)+\frac{k-1}m\,\big|\gamma\,E^{\gamma-1}\,\wv^{\beta-2}\,v -\varepsilon\,x\big|^2\)\\
&\le k\,m^{k-1}\(C\,\one_{B_R}- \frac\varepsilon2\(\wv^{\gamma\,\beta+\beta-2}+\wx^\alpha\)+\,\frac{k-1}m\,\big(\gamma^2\,E^{2\,\gamma-2}\,\wv^{2\,\beta-2}+\varepsilon^2\,\wx^2\big)\)
\end{align*}
as $\gamma\le 2$. Thanks to~\eqref{equiv1}, we have $m\gtrsim E^\gamma$ and $m\gtrsim \wx^2$, so
\[
\frac{k-1}m\(\gamma^2\,E^{2\,\gamma-2}\,\wv^{2\,(\beta-1)}+\varepsilon^2\,\wx^2\)\lesssim E^{\gamma-2}\,\wv^{2\,\beta-2} + \varepsilon^2\,\lesssim \wv^{\gamma\,\beta-2}+ \varepsilon^2 .
\]
With the fact $\gamma\,\beta+\beta -2 >\gamma\,\beta -2$, this allows us to get
\be{plug}
\cL^*(m^k)\le C\,\one_{B_R}-\frac\varepsilon4\,m^{k-1}\(\wv^{\gamma\,\beta+\beta-2}+\wx^\alpha\).
\ee
Finally notice that
\[
m\lesssim \wv^{\gamma\,\beta}+\wx^{\max\{\gamma\,\alpha,2\}}\lesssim \(\wv^{\gamma\,\beta+\beta-2}+\wx^\alpha\)^{\max\{\frac{\gamma\,\beta}{\gamma\,\beta+\beta-2}, \gamma , \frac2{\alpha}\}}
\]
and
\[\textstyle
\max\left\{\frac{\gamma\,\beta}{\gamma\,\beta+\beta-2}, \gamma , \frac2{\alpha}\right\}=\max\left\{ \gamma , \frac2{\alpha}\right\}=\max\left\{1,\frac2{\beta}, \frac2{\alpha}\right\}
\]
which plugged in~\eqref{plug} proves that~\eqref{Eqn:Lyapunov} holds with $\ell=1-\min\left\{1,\alpha/2,\beta/2\right\}$.
\end{proof}

\subsection{A Lyapunov function when \texorpdfstring{$\alpha<1 $}{alpha<1} and \texorpdfstring{$\beta\le1$}{beta<=1}}\label{Sec:Lyapunov2}~

\noindent Building a Lyapunov function is delicate: see Proposition~\ref{prop:main} below for the statement.

\subsubsection{Definition of the Lyapunov function}\label{Sec:Lyapunov-Notation}
With the choice
\[
0<\rma<\min\left\{\frac1\beta,\frac1\alpha-1\right\}\,,\quad\rho=\frac{\rma\,\beta}{1+\beta}\,,
\]
we notice that
\be{eq:parameter-consequences}
\rma\,\beta<1\,,\quad \rma<\frac1\alpha-1\,,\quad \rma\,(1-\beta)<\frac1\alpha-1\,,\quad0<\rho<1\,.
\ee
Set
\[
\Theta(x,v):=\frac{\wv}{E(x,v)^\rma}\,.
\]
Let $\eta$ and $\chi$ be cut-off functions defined as follows. We assume that $\chi\in \mathrm C^2([0,+\infty))$ is non-increasing and satisfies
\[
0\le\chi\le1\,,\quad\chi=1\quad\hbox{on }[0,2]\,,\quad\chi=0\quad\hbox{on }[3,+\infty)
\]
while $\eta\in \mathrm C^\infty([0,+\infty))$ is a non-increasing function such that
\[
0\le\eta\le1\,,\quad\eta=1\quad\hbox{on }[0,1]\,,\quad\eta=0\quad\hbox{on }[2,+\infty)\,.
\]
For $s>0$, we set
\[
\vartheta(s):=\frac{\eta(s)}{s}\,,\quad\zeta_\lambda(s):=e^{-\lambda\,\vartheta(s)}\,.
\]
Then $\vartheta'\le0$ and hence $\zeta_\lambda'\ge0$. Moreover, we notice that
\begin{align*}
&\zeta_\lambda(s)=e^{-\lambda/s}\,,&&0<s\le1\,,\\
&c_\eta\le\zeta_\lambda(s)\le1\,,\quad|\zeta_\lambda'(s)|+|\zeta_\lambda''(s)|\le C_\eta\,\lambda\,,&&1\le s\le2\,,\\
&\zeta_\lambda(s)=1\,,&&s\ge2\,,
\end{align*}
where $c_\eta>0$ and $C_\eta>0$ are constants independent of $\lambda\in(0,1]$, so that
\be{zeta:der}
\zeta_\lambda'(s)=\lambda\,\frac{\zeta_\lambda(s)}{s^2}\,,\quad\zeta_\lambda''(s)=\(\frac{\lambda^2}{s^4}-\frac{2\,\lambda}{s^3}\)\zeta_\lambda(s)\le\lambda^2\,\frac{\zeta_\lambda(s)}{s^4}\,,\quad0<s\le1\,.
\ee
Define
\be{eq:components}
m:=M_0+M_E\quad\mbox{where}\quad
M_0=e^{\frac{\kappa}{\alpha}F}\,,\quad M_E=\zeta_\lambda(\Theta)\,W\,,\quad W=e^{\kappa\,E}
\ee
and
\be{eq:F}
F(x,v):=\wx^\alpha +\alpha\,\chi\big(\Theta(x,v)\big)\,\wx^{\alpha-2}\,x\cdot v\,.
\ee
With the above definitions, we can state the following \emph{weak Foster-Lyapunov estimate}.
\begin{proposition}\label{prop:main} Let $(\alpha,\beta)\in(0,1)\times(0,1]$. There exist positive constants $\lambda_0$, $\kappa_0$ and $c_0$, depending only on $\alpha$, $\beta$, $\rma$, $d$ and on the cutoff profiles, such that the following holds. For every $\lambda\in(0,\lambda_0]$ and $\kappa\in(0,\kappa_0\,\lambda]$, there exist $C>0$ and $R>0$ for which $m$ is positive, coercive, and satisfies
\be{eq:weak-main}
\cL^*m
\le C\,\one_{\{E\le R\}}
-c_0\,\kappa^{1+\sigma_*}\,
\frac{m}{\{\ln(e+m)\}^{\sigma_*}}
\ee
with
\[\textstyle
\sigma_*=2\,\max\left\{\frac1\alpha-1,\frac1\beta-1\right\}\,,
\]
The radius $R$ and the compact-set constant $C$ may depend on
$\lambda$ and $\kappa$.
\end{proposition}
The proof is postponed to Section~\ref{Sec:PropMain}.

\subsubsection{Technical estimates on \texorpdfstring{$\Theta$ and $\sfL^\star$}{ThetaAndLstar}}
Let us start with some elementary computations.
\begin{lemma}\label{lem:propTheta} Let $(\alpha,\beta)\in(0,1)\times(0,1]$. With the notation of Section~\ref{Sec:Lyapunov-Notation} and $E=E(x,v)$, one has
\be{eq:T-z}
\sfT \Theta(x,v)=-\,E^{-\rma}\,\wx^{\alpha-2}\,\frac{x\cdot v}{\wv}\,,
\ee
\be{eq:L-z}
\begin{aligned}
\sfL^\star\Theta(x,v)=&\,E^{-\rma}\,\sfL^\star(\wv)-\rma\,\wv\,E^{-\rma-1}\,\sfL^\star E\\
&\,+\,\rma\,(\rma+1)\,E^{-\rma-2}\,\wv^{2\,\beta-3}\,|v|^2-2\,\rma\,E^{-\rma-1}\,\wv^{\beta-3}\,|v|^2\,.
\end{aligned}
\ee
\end{lemma}
\begin{proof} By applying the product rule to $\Theta=\wv\,E^{-\rma}$, we obtain
\[
\sfT \Theta=\wv\,\sfT E^{-\rma} + E^{-\rma}\,\sfT (\wv)=E^{-\rma}\,\sfT (\wv)=-\,E^{-\rma}\,\wx^{\alpha-2}\,\frac{x\cdot v}{\wv}
\]
after recalling that $\sfT(E^{-\rma})=0$. Similarly, from
\[
\sfL^\star\Theta=\sfL^\star(\wv\,E^{-\rma})=\wv\,\sfL^\star(E^{-\rma}) + \sfL^\star(\wv)\,E^{-\rma} + 2\,\nabla_v\wv \cdot \nabla_v E^{-\rma}
\]
and
\[
\sfL^\star(E^{-\rma})=-\,\rma\,E^{-\rma-1}\,\sfL^\star E+\rma\,(\rma+1)\,E^{-\rma-2}\,|\nabla_vE|^2\,,
\]
we obtain~\eqref{eq:L-z} by the product rule, and using $\nabla_vE=\wv^{\beta-2}\,v$.
\end{proof}
Next we establish \emph{zonal} estimates for $\Theta$.
\begin{lemma}\label{lem:zonesTheta} Let $(\alpha,\beta)\in(0,1)\times(0,1]$. With the notation of Section~\ref{Sec:Lyapunov-Notation}, there exists a positive constant $R_\Theta>0$ such that, for any $E\ge R_\Theta$, the following
properties hold:
\begin{align}
\label{eq:z-leq-3}
E\asymp\wx^\alpha\quad\mbox{and}\quad|\nabla_v\Theta|+\big|\cL^\star\Theta\big|\lesssim E^{-\rma}\quad\mbox{if}\quad\Theta\le3\,,\\
\label{eq:z-negative}
\cL^\star\Theta \le-\,\frac14\,E^{-\rma}\,\wv^{\beta-1}\quad\mbox{if}\quad E^{-\rho}\le\Theta\le3\,.
\end{align}
\end{lemma}
\begin{proof}~We prove the two estimates successively.

\smallskip\noindent\textbf{$\bullet$ Proof of~\eqref{eq:z-leq-3}.} Assume first that $\Theta\le3$. Since $\wv=\Theta\,E^\rma$,
\[
\wv^\beta\le3^\beta\,E^{\,\rma\,\beta}=o(E)\,,
\]
because $\rma\,\beta<1$. Hence $E\asymp\wx^\alpha$. Since,
\[
\nabla_v\Theta=E^{-\rma-1}\(E-\rma\,\wv^{\beta}\) \frac{v}{\wv}\,,
\]
one obtains immediately that
\[
|\nabla_v\Theta|\lesssim E^{-\rma}\(1+\frac{\wv^\beta}{E}\)\le E^{-\rma}\(1+E^{\,\rma\,\beta-1}\)\lesssim E^{-\rma}\,.
\]
Next, taking into account
\be{eq:L-wv}
\sfL^\star(\wv)=\wv^{-1}\(d-1-\wv^{\beta}\) +\wv^{-3}\(\wv^{\beta}+1\),
\ee
\[
\sfL^\star E=\wv^{\beta-2}\(d+\( (\beta-2)-\wv^{\beta} \) \frac{|v|^2}{\wv^2}\),
\]
we obtain
\[
\big|\sfL^\star(\wv)\big|\lesssim \wv^{\beta-1}\quad\mbox{and}\quad\big|\sfL^\star E\big|\lesssim \wv^{2\,\beta-2}\,.
\]
Consequently, by~\eqref{eq:L-z}, we likewise obtain (remember that $\beta-1\le 0$, $2\,\beta-1\le 1$),
\begin{multline*}
\big|\sfL^\star\Theta\big|\lesssim E^{-\rma}\,\wv^{\beta-1}+\wv\,E^{-\rma-1}\,\wv^{2\,\beta-2}\\
+\,\wv\,E^{-\rma-2}\,\wv^{2\,\beta-2}+ E^{-\rma-1}\,\wv^{\beta-1}\lesssim E^{-\rma}\,,
\end{multline*}
because $\wv^\beta\lesssim E^{\,\rma\,\beta}$. Finally,
\[
|\sfT \Theta|\le E^{-\rma}\,\wx^{\alpha-1}\le E^{-\rma}\,,
\]
since $\alpha<1$ and $\wx\ge1$. This proves~\eqref{eq:z-leq-3}.

\medskip\noindent \textbf{$\bullet$ Proof of~\eqref{eq:z-negative}.} We now assume $E^{-\rho}\le\Theta\le3$. Since
\[
\wv=\Theta\,E^{\,\rma}\ge E^{\,\rma-\rho}=E^{\,\frac{\rho}\beta}\,,
\]
we have $\wv\to+\infty$ uniformly in this region.

Observe that, by~\eqref{eq:L-wv},
\begin{align*}
\sfL^\star(\wv)&=\wv^{\beta-1}\((d-1)\,\wv^{-\beta} -1\) +\wv^{\beta-3}\(1+\wv^{-\beta}\),\\
&\le-\,\wv^{\beta-1}+ \wv^{\beta-1}\,\Big((d-1)\,\wv^{-\beta} +\wv^{-2}\(1+\wv^{-\beta}\)\Big)\le-\, \frac12\,\wv^{\beta-1}
\end{align*}
when $|v|$ is large. Then, since
\[
\big|\sfL^\star E\big|\le\wv^{\beta-2}\( d+ |\beta-2|+\wv^{\beta} \)\lesssim \wv^{2\,\beta-2}\,,
\]
we have
\[
\left|\rma\,\wv\,E^{-\rma-1}\,\sfL^\star E\right|\le \rma\,E^{-\rma-1}\,\wv^{2\,\beta-1}\,.
\]
For the transport part, using~\eqref{eq:z-leq-3}, we obtain
\[
|\sfT \Theta|\le\left| E^{-\rma}\,\wx^{\alpha-2}\,\frac{x\cdot v}{\wv} \right|\lesssim E^{1-\frac1{\alpha}-\rma}\,.
\]
Combining~\eqref{eq:L-z} and~\eqref{eq:T-z}, we can find a constant $C>0$ such that
\begin{align*}
\cL^\star\Theta &=\sfT \Theta + \sfL^\star\Theta\\
&\le|\sfT \Theta| - \frac12\,E^{-\rma}\,\wv^{\beta-1} +\left|\rma\,\wv\,E^{-\rma-1}\,\sfL^\star E\right| + \rma\,(\rma+1)\,E^{-\rma-2}\,\wv^{2\,\beta-3}\,|v|^2\\
&\le-\, \frac12\,E^{-\rma}\,\wv^{\beta-1} +C\big(E^{1-\frac1{\alpha}-\rma} + E^{-\rma-1}\,\wv^{2\,\beta-1} + E^{-\rma-2}\,\wv^{2\,\beta-1}\big)\\
&\le-\, \frac12\,E^{-\rma}\,\wv^{\beta-1} + C \,E^{-\rma}\,\wv^{\beta-1}\(\wv^{1-\beta}\,E^{1-\frac1{\alpha}} + \wv^{\beta}\,E^{-1} +E^{-2}\,\wv^{\beta}\).
\end{align*}
Since $\Theta\le3$,
\[
\wv^{1-\beta}\,E^{1-\frac1{\alpha}}=\Theta^{1-\beta}\,E^{\,\rma\,(1-\beta)+1-\frac1{\alpha}}\le 3^{1-\beta}\,E^{\,\rma\,(1-\beta)+1-\frac1{\alpha}}, \quad\frac{\wv^\beta}{E}\lesssim E^{-(1-\rma\,\beta)}\,,
\]
and $\rma\,(1-\beta)+1-1/\alpha<0$ and $1-\rma\,\beta >0$, we obtain~\eqref{eq:z-negative} after increasing $R_\Theta$.
\end{proof}

\subsubsection{The component \texorpdfstring{$M_0$}{M0}}
We start with an exact \emph{macro--micro} computation.
\begin{lemma}\label{lem:macro} Let $(\alpha,\beta)\in(0,1)\times(0,1]$. With $F$ defined by~\eqref{eq:F}, one
has
\begin{align}
\sfT F={}& \alpha\,\wx^{\alpha-2}\,\wv^{\beta-2}\,x\cdot v +\alpha\,\chi(\Theta)\,\wv^{\beta-2}\wx^{\alpha-2}\, \big( |v|^2 +(\alpha-2)\,\wx^{-2}\,(x\cdot v)^2\big) \notag\\
&-\alpha\,\chi(\Theta)\,\wx^{2\,\alpha-4}\,|x|^2 +\alpha\,\chi'(\Theta)\,(\sfT \Theta)\,\wx^{\alpha-2}\,x\cdot v\,, \label{eq:TF}\\
\sfL^\star F={}& -\alpha\,\chi(\Theta)\,\wx^{\alpha-2}\,x\cdot\wv^{\beta-2}\,v +\alpha\,\chi'(\Theta)\,(\sfL^\star \Theta)\,\wx^{\alpha-2}\,x\cdot v \notag\\
&+\alpha\,\chi''(\Theta)\,\wx^{\alpha-2}\,(x\cdot v)\,|\nabla_v\Theta|^2 +2\,\alpha\,\chi'(\Theta)\,\wx^{\alpha-2}\,x\cdot\nabla_v\Theta\,,
\label{eq:LF}
\end{align}
and finally,
\begin{align}
\cL^*F={}&
\alpha(1-\chi(\Theta))\,\wx^{\alpha-2}\, \wv^{\beta-2}\,x\cdot v \notag\\
&+\alpha\,\chi(\Theta)\,\Big( \wv^{\beta-2}\,\wx^{\alpha-2}\(|v|^2 +(\alpha-2)\,\wx^{-2}\,(x\cdot v)^2\) -\wx^{2\,\alpha-4}\,|x|^2\Big) \notag\\
&+\alpha\,\chi'(\Theta)\,\wx^{\alpha-2}\,\big((\cL^*\Theta)\,x\cdot v +2 x\cdot\nabla_v\Theta \big)\notag\\
&+\alpha\,\chi''(\Theta)\,\wx^{\alpha-2}\,(x\cdot v)\,|\nabla_v\Theta|^2\,.
\label{eq:LstarF}
\end{align}
\end{lemma}
\begin{proof} The identities follow from the product and chain rules. Start from
\[
\sfT F = \sfT\big(\wx^\alpha\big) +\alpha \sfT \Big(\chi\big(\Theta(x,v)\big)\,\wx^{\alpha-2}\,x\cdot v\Big)\,,
\]
We compute each term separately as
\[
\sfT(\wx^\alpha) =\alpha\,\wx^{\alpha-2}\,x\cdot\wv^{\beta-2}\,v
\]
and
\begin{align*}
\sfT \big(\chi(\Theta)&\wx^{\alpha-2}\,x\cdot v\big) \\
&= \sfT \(\chi(\Theta)\)\,\wx^{\alpha-2}\,x\cdot v + \chi(\Theta)\,\wx^{\alpha-2}\, \sfT \(x\cdot v\) + \chi(\Theta) \sfT \(\wx^{\alpha-2}\) x\cdot v\\
&=\chi'(\Theta)\,(\sfT\Theta)\,\wx^{\alpha-2}\,x\cdot v + \chi(\Theta)\,\wx^{\alpha-2} \( \wv^{\beta-2}\, | v|^2 - \wx^{\alpha-2}\, | x |^2\) \\
&\qquad \qquad + \chi(\Theta)\,(\alpha-2)\,\wx^{\alpha-4} \wv^{\beta-2}\,\,(x\cdot v)^2\,.
\end{align*}
Gathering these computations gives~\eqref{eq:TF}. Hence
\begin{align*}
\sfL^\star F&=\alpha \wx^{\alpha-2} \( \sfL^\star(\chi(\Theta))\,(x\cdot v) + \chi(\Theta)\,\sfL^\star(x\cdot v) + 2 \chi'(\Theta)\,\nabla_v\Theta\cdot x \)\\
&=-\,\alpha\,\chi(\Theta)\,\wx^{\alpha-2}\,x\cdot\wv^{\beta-2}\,v +\alpha\,\chi'(\Theta)\,(\sfL^\star \Theta)\,\wx^{\alpha-2}\,x\cdot v \notag\\
&+\alpha\,\chi''(\Theta)\,\wx^{\alpha-2}\,(x\cdot v)\,|\nabla_v\Theta|^2 +2\,\alpha\,\chi'(\Theta)\,\wx^{\alpha-2}\,x\cdot\nabla_v\Theta\,,
\end{align*}
which is~\eqref{eq:LF}. Equation~\eqref{eq:LstarF} follows by combining $\sfT F$ and $\sfL^\star F$.
\end{proof}
\begin{proposition}\label{prop:LM0} Let $(\alpha,\beta)\in(0,1)\times(0,1]$. With the notation of Section~\ref{Sec:Lyapunov-Notation}, there are positive constants $c_M$, $C_M$, and $R_M$ such that, for any $\kappa\in(0,1/4]$,
\be{eq:M0-global}
\cL^*M_0 \le C_M\,\one_{\{E\le R_M\}} -c_M\kappa\,M_0\,\chi(\Theta)\,\wx^{2\,\alpha-4}\,|x|^2 +C_M\,\kappa\,M_0\,\one_{\{\Theta\ge2\}}\,.
\ee
\end{proposition}
\begin{proof} Recall that
\be{eq:Lstar-M0-exact}
\frac{\cL^*M_0}{M_0} =\frac\kappa\alpha\,\cL^*F +\frac{\kappa^2}{\alpha^2}\,|\nabla_vF|^2\,.
\ee
Furthermore, since
\[
\nabla_v\Theta =E^{-\rma-1} \(E-\rma\,\wv^{\beta}\) \frac{v}{\wv}\,,
\]
we have
\begin{align*}\label{eq:grad-F}
\nabla_vF &=\alpha\,\chi(\Theta)\,\wx^{\alpha-2}\,x +\alpha\,\chi'(\Theta)\,\wx^{\alpha-2}\,(x\cdot v)\,\nabla_v\Theta\\
&=\alpha\,\chi(\Theta)\,\wx^{\alpha-2}\,x +\alpha\,\chi'(\Theta)\,\wx^{\alpha-2}\,(x\cdot v)E^{-\rma-1} \(E -\rma\,\wv^{\beta}\) \frac{v}{\wv}\,.
\end{align*}

We now estimate~\eqref{eq:Lstar-M0-exact} in three different regions.

\smallskip\noindent\textbf{$\bullet$ Range $\Theta\le2$.} Here $\chi(\Theta)=1$ and $\chi'(\Theta)=\chi''(\Theta)=0$. Since $\alpha-2<0$, by Lemma~\ref{lem:macro},
\be{eq:F-drift-zone-I}
\cL^*F \le\alpha\,\wx^{\alpha-2}\,\wv^\beta -\alpha\,\wx^{2\,\alpha-4}\,|x|^2\,.
\ee
In this range, by Lemma~\ref{lem:zonesTheta}, $E\asymp\wx^\alpha$ and
$\wv^\beta\lesssim\,E^{\,\rma\,\beta}$. Hence
\[
\frac{\wx^{\alpha-2}\,\wv^\beta} {\wx^{2\,\alpha-4}\,|x|^2} \lesssim \frac{\wv^\beta}{\wx^\alpha} \lesssim\,E^{\,\rma\,\beta-1}\,.
\]
Since $\rma\,\beta-1<0$, after increasing $R_M$ if necessary,
\[
\cL^*F\le-\,\frac\alpha2\,\wx^{2\,\alpha-4}\,|x|^2\quad \text{on } \{E>R_M\}\,.
\]
Moreover,
\[
|\nabla_vF|^2 =\alpha^2\,\wx^{2\,\alpha-4}\,|x|^2\,.
\]
Thus, for $\kappa\le1/4$,
\be{eq:M0-zone-I}
\cL^*M_0 \le-\,\frac\kappa4\,M_0\,\wx^{2\,\alpha-4}\,|x|^2 \quad \text{on } \{E>R_M\}\,.
\ee

\smallskip\noindent\textbf{$\bullet$ Range $2\le\Theta\le3$.} Since $E\asymp\wx^\alpha$ and $\wv=\Theta E^\rma$,
\be{eq:mixed-bounded}
\big|\wx^{\alpha-2}\,x\cdot v\big| \le\wx^{\alpha-1}\,\wv \lesssim E^{1+\rma-\frac1\alpha}\,.
\ee
Also notice $\big|\wx^{\alpha-2}\,x\big|\le1$, $\big|\wv^{\beta-2}\,v\big|\le1$,
$\wx^{2\,\alpha-4}\,|x|^2\le1$, and, as a consequence,
\[
\wx^{\alpha-2}\,\wv^\beta \lesssim E^{1-\frac{2}{\alpha}+\rma\,\beta}\,.
\]
Together with~\eqref{eq:z-leq-3}, these estimates imply
\[
\label{eq:F-transition-bounded}
|\cL^*F|+|\nabla_vF|^2\lesssim 1\quad\mbox{and}\quad
|\cL^*M_0|\lesssim \kappa\,M_0\,.
\]

\smallskip\noindent\textbf{$\bullet$ Range $\Theta\ge3$.}
Here $\chi=0$, hence $F=\wx^\alpha$ and $\nabla_vF=0$. Since $\big|\wx^{\alpha-2}\,x\big|\le1$ and $\big|\wv^{\beta-2}\,v\big|\le1$,
\[
|\cL^*F| =\alpha\,\big|\wx^{\alpha-2}\,x\cdot\wv^{\beta-2}\,v\big| \le\alpha\,.
\]
Therefore
\be{eq:M0-zone-III}
|\cL^*M_0|\le\kappa\,M_0\,.
\ee

Combining~\eqref{eq:M0-zone-I} and~\eqref{eq:M0-zone-III}, and using $0\le\chi\,\wx^{2\,\alpha-4}\,|x|^2\le1$ on $\Theta\ge2$, completes the proof of~\eqref{eq:M0-global}. \end{proof}

\subsubsection{Estimates on \texorpdfstring{$M_E$}{ME}}
\begin{lemma}\label{lem:ME}
Let $W$ and $M_E$ be defined by~\eqref{eq:components}. One has
\be{eq:LstarME}
\cL^\star\,M_E=\(\frac{\sfL^\star W}{W}+\frac{\zeta_\lambda'(\Theta)}{\zeta_\lambda(\Theta)}\(\cL^\star\Theta +2\,\kappa\,\wv^{\beta-2}\,v \cdot\nabla_v\Theta\)+\frac{\zeta_\lambda''(\Theta)}{\zeta_\lambda(\Theta)}\,|\nabla_v\Theta|^2\)M_E\,.
\ee
\end{lemma}
\begin{proof} Start from
\begin{align*}
&\cL^\star M_E=\sfT M_E+ \sfL^\star\,M_E\\
&=\kern-0.25pt W\zeta_\lambda'(\Theta)\sfT \Theta+\zeta_\lambda(\Theta)\,\sfL^\star W+W\zeta_\lambda'(\Theta)\,\sfL^*\Theta+W\zeta_\lambda''(\Theta)\,|\nabla_v\Theta|^2+2\,\zeta_\lambda'(\Theta)\,\nabla_v\Theta\cdot\nabla_vW
\end{align*}
using the chain rule, and the product rule for the second-order operator $\sfL^\star$. Since $\nabla_vW=\kappa\,W\,\nabla_vE$, we conclude,
\[
\cL^\star M_E=\(\frac{\sfL^\star W}{W}+\frac{\zeta_\lambda'(\Theta)}{\zeta_\lambda(\Theta)}\,\cL^\star\Theta+\frac{\zeta_\lambda''(\Theta)}{\zeta_\lambda(\Theta)}\,|\nabla_v\Theta|^2+2\,\kappa\,\frac{\zeta_\lambda'(\Theta)}{\zeta_\lambda(\Theta)}\,\nabla_v\Theta\cdot v\,\wv^{\beta-2}\)M_E\,,
\]
which is nothing else than~\eqref{eq:LstarME}.
\end{proof}
\begin{proposition}\label{prop:energy} Let $(\alpha,\beta)\in(0,1)\times(0,1]$. With the notation of Section~\ref{Sec:Lyapunov-Notation}, if $c_M$ is the constant in Proposition~\ref{prop:LM0}, there are positive constants $C_E$, $c_E$ and $R_E$ such that
\be{eq:ME-global}
\cL^*M_E \le C_E\,\one_{\{E\le R_E\}} -c_E\,\kappa\,M_E\,\wv^{2\,\beta-2} +\frac12\,c_M\,\kappa\,M_0\,\chi(\Theta)\,\wx^{2\,\alpha-4}\,|x|^2\,.
\ee
\end{proposition}
\begin{proof} A direct computation shows that
\begin{align*}
\frac{\sfL^\star W}{W}&=\kappa\,\sfL^\star E+\kappa^2\,|\nabla_vE|^2\\
&=\kappa\,\Big(d\wv^{\beta-2}+(\beta-2)\,|v|^2\,\wv^{\beta-4}-(1-\kappa)\,|v|^2\,\wv^{2\,\beta-4}\Big)
\end{align*}
and, as a consequence,
\be{eq:LW/W}
\frac{\sfL^\star W}W\lesssim -\,\kappa\,\wv^{2\,\beta-2}
\ee
for $\wv$ sufficiently large and $0<\kappa\le1/2$. We consider four different regimes.

\smallskip\noindent\textbf{$\bullet$ Range $\Theta\le E^{-\rho}$.}
For large $E$, one has $\Theta\le1$, $\chi=1$ and $\zeta_\lambda(\Theta) = e^{-\lambda/\Theta}$. Therefore
\[
\ln\(\frac{M_E}{M_0}\)=\kappa\,E -\frac\lambda \Theta - \frac\kappa\alpha\,\wx^\alpha -\kappa\,\wx^{\alpha-2}\,x\cdot v =\kappa\,\frac{\wv^\beta}\beta -\frac\lambda \Theta -\kappa\,\wx^{\alpha-2}\,x\cdot v\,.
\]
Thanks to the first estimate in~\eqref{eq:z-leq-3}, we have
\[
\label{eq:mixed-small-z}
\big|\wx^{\alpha-2}\,x\cdot v\big| \le\wx^{\alpha-1}\,\wv \lesssim E^{1-\frac{1}{\alpha}+\frac{\rho}\beta} \lesssim 1
\]
because $1-{1}/{\alpha}+{\rho}/{\beta} < 0$. We deduce from $\Theta \le E^{-\rho}$ and $\wv \le E^{\,\rma-\rho} = E^{{\rho}/{\beta}}$ that $\wv^\beta\le E^{\rho}$. Together with $1/\Theta\ge E^\rho$, we arrive at
\[
\ln\(\frac{M_E}{M_0}\) -\(\frac\kappa\beta-\lambda\)E^\rho\lesssim \kappa\,.
\]
Under the condition $\Theta\le E^{-\rho}$, choosing $\kappa/\lambda\le\kappa_0\le\beta/2$ gives
\be{eq:ME-M0-ratio}
\frac{M_E}{M_0}\lesssim e^{-\,\frac12\,\lambda\,E^\rho}\,.
\ee

To estimate $\cL^\star M_E$, it remains to quantify the derivatives. For $E$ large enough, we have $\Theta <1$ and read from~\eqref{zeta:der} that $\zeta_\lambda'/\zeta_\lambda=\lambda/\Theta^2$ and $\zeta_\lambda''/\zeta_\lambda\le\lambda^2/\Theta^4$. We get
\begin{align*}
\cL^\star M_E& \le\Big(-\,\kappa\,\wv^{2\,\beta-2} + \lambda\,\Theta^{-2}\(\cL^\star\Theta +2\,\kappa\,\wv^{\beta-2}\, v \cdot\nabla_v\Theta\) +\lambda^2\,\Theta^{-4}\,|\nabla_v\Theta|^2 \Big)\,M_E
\end{align*}
after recalling~\eqref{eq:LW/W}. By Lemma~\ref{lem:zonesTheta}, we have $|\nabla_v\Theta|+|\cL^*\Theta|\lesssim E^{-\rma}$ and
\[
\cL^\star M_E\le\Big(-\,\kappa\,\wv^{2\,\beta-2} + \lambda\,\Theta^{-2}\,E^{-\rma} +\lambda^2\,\Theta^{-4}\,E^{-2\,\rma} \Big)\,M_E\,.
\]
Since $\wv\ge1$, $\Theta=\wv \,E^{-\rma}\ge E^{-\rma}$, we obtain $\Theta^{-1}\le E^\rma$ and
\[
\frac{|\cL^*M_E|}{M_E} \le C_{\lambda,\kappa}\big(1+E^\rma+E^{2\,\rma}\big)\,,
\]
which, under the condition $\Theta\le E^{-\rho}$, proves
\be{eq:ME-polynomial-bound}
|\cL^*M_E|\le C_{\lambda,\kappa}\,(1+E^{2\,\rma})\,M_E\,.
\ee
Since $E\asymp\wx^\alpha$ and $\wx^{2\,\alpha-4}\,|x|^2 \asymp E^{-{2\,(1-\alpha)}/{\alpha}}$. Combining~\eqref{eq:ME-M0-ratio} and~\eqref{eq:ME-polynomial-bound}, we obtain that
\[
\frac{\cL^*M_E +c_E\,\kappa\,M_E\,\wv^{2\,\beta-2}} {\kappa\,M_0\,\wx^{2\,\alpha-4}\,|x|^2} \le C_{\lambda,\kappa} \,E^{2\,\rma+2\,\frac{1-\alpha}{\alpha}}\,e^{-c\,\lambda\,E^\rho}
\]
is uniformly small outside of a compact set in $\RRd$ for some $c_E>0$. For some $R>R_M$, this proves
\be{eq:ME-small-absorption}
(\cL^*M_E)_+ +c_E\,\kappa\,M_E\,\wv^{2\,\beta-2} \le\frac12\,c_M\,\kappa\,M_0\,\chi(\Theta)\,\wx^{2\,\alpha-4}\,|x|^2\,,
\ee
where $R_M$ and $c_M$ are as in Proposition~\ref{prop:LM0}.

\smallskip\noindent\textbf{$\bullet$ Range $E^{-\rho}\le\Theta\le1$.} In this case, we shall use Lemma~\ref{lem:zonesTheta} to get
\[
\cL^\star\Theta\le-\,c\,\wv^{\beta-1}\,E^{-\rma}\,,\quad|\nabla_v\Theta|\le C_{\nabla}\,E^{-\rma}\,,
\]
with $c=1/4$ and $C_{\nabla}=1+\rma\,3^\rma$. Thus, recalling~\eqref{eq:LW/W},
\begin{align*}
&\cL^\star M_E + \kappa\,\wv^{2\,\beta-2}\,M_E\\
&\lesssim\( -c\,\lambda\,\Theta^{-2}\,E^{-\rma}\,\wv^{\beta-1} +2\,\kappa\,\lambda\,C_{\nabla}\,\Theta^{-2}\,\wv^{\beta-1}\,E^{-\rma}+\lambda^2\,C_{\nabla}^2\,\Theta^{-4}\,E^{-2\,\rma}\)M_E\\
&\lesssim \lambda\,\Theta^{-2}\,E^{-\rma}\,\wv^{\beta-1}\( -c+2\,\kappa\,C_{\nabla}+\lambda\,C_{\nabla}^2\,\Theta^{-2}\,E^{-\rma}\,\wv^{1-\beta}\)M_E\,.
\end{align*}
Observe crucially that
\[
\Theta^{-2}\,E^{-\rma}\,\wv^{1-\beta}=\Theta^{-2}\,E^{-\rma}(\Theta\,E^{\,\rma})^{1-\beta}=\Theta^{-(1+\beta)}\,E^{-\rma\,\beta}\le E^{(1+\beta)\,\rho}\,E^{-\rma\,\beta}=1\,,
\]
thanks to the choice of $\rho$. As a consequence, as long as $2\,\kappa\,C_{\nabla} + \lambda\,C_{\nabla}^2\le c/2$,
\[
\cL^\star M_E + \kappa\,\wv^{2\,\beta-2}\,M_E\lesssim -\,c \,\lambda\,\Theta^{-2}\,E^{-\rma}\,\wv^{\beta-1}\,M_E\,.
\]
If $E^{-\rho}\le \Theta\le1$, this proves that
\be{eq:ME-zone-A}
\cL^*M_E \le-\,c_E\,\lambda\,E^{-\rma}\,\wv^{\beta-1}\,M_E -c_E\,\kappa\,\wv^{2\,\beta-2}\,M_E\,.
\ee

\smallskip\noindent\textbf{$\bullet$ Range $1\le\Theta\le2$.}
Here, $\wv=\Theta\,E^\rma\asymp E^\rma$ so that $\wv$ is large when $E$ is large. Additionally, $\zeta_\lambda$ is bounded below, and $|\zeta_\lambda'|+|\zeta_\lambda''|\lesssim\lambda$. Since $\zeta_\lambda'\ge0$ and $\cL^*\Theta<0$ for large $E$, the term $(\cL^\star\Theta)\,\zeta_\lambda'(\Theta)/\zeta_\lambda(\Theta)$ is non-positive and may be discarded. Thus, recalling~\eqref{eq:LstarME},~\eqref{eq:LW/W}, and $|\nabla_v\Theta|\lesssim E^{-\rma}$,
\begin{align*}
\cL^\star M_E + &\kappa\,\wv^{2\,\beta-2}\,M_E \lesssim\( \kappa\,\lambda \wv^{\beta-1}\,|\nabla_v\Theta| +\lambda\,|\nabla_v\Theta|^2 \)M_E \\
&\lesssim \lambda \( \kappa\,E^{\,\rma\,(\beta-1)}\,E^{-\rma} + E^{-2\,\rma} \) M_E \lesssim \kappa\,\lambda\,E^{\,\rma\,(\beta-2)} \lesssim \kappa\,\lambda\,\wv^{\beta-2}\, M_E\,,
\end{align*}
when $E$ is large. After increasing $R$, this proves
\be{eq:ME-zone-C1}
\cL^*M_E\le-\,c_E\,\kappa\,M_E\,\wv^{2\,\beta-2}\quad\mbox{if}\quad1\le\Theta\le2\,.
\ee

\smallskip\noindent\textbf{$\bullet$ Range $\Theta\ge2$.}
Here $\zeta_\lambda=1$, hence $M_E=W$. In this range, $\wv\ge2\,E^\rma$ which ensures that $\wv$ is large as $E$ is large, and it follows directly from~\eqref{eq:LW/W} and~\eqref{eq:LstarME} that
\be{eq:ME-zone-C2}
\cL^*M_E\le-\,c_E\,\kappa\,M_E\,\wv^{2\,\beta-2}\quad\mbox{if}\quad\Theta\ge2\,.
\ee

\medskip Altogether, there are positive constants $\lambda_0$, $\kappa_0$ and $c_E>0$ such that, for any $\lambda\in(0,\lambda_0]$ and $\kappa\in(0,\kappa_0\,\lambda]$, we obtain~\eqref{eq:ME-global} as a consequence of~\eqref{eq:ME-small-absorption},~\eqref{eq:ME-zone-A},~\eqref{eq:ME-zone-C1} and~\eqref{eq:ME-zone-C2} for $E$ sufficiently large.
\end{proof}

\subsubsection{Proof of Proposition~\ref{prop:main}}\label{Sec:PropMain}

We are now ready to prove the main result of this section. Adding~\eqref{eq:M0-global} and~\eqref{eq:ME-global}, and taking a larger $R$ if necessary, gives
\begin{multline}\label{eq:combined-before-high-z}
\cL^\star m-C\,\one_{\{E\le R\}}\\
\le-\frac12\,c_M\kappa\,\chi(\Theta)\,M_0\,\wx^{2\,\alpha-4}\,|x|^2+C_M\,\kappa\,M_0\,\one_{\{\Theta\ge2\}}-c_E\,\kappa\,M_E\,\wv^{2\,\beta-2}.
\end{multline}
It remains to absorb the positive term in the righ-hand side. If $\Theta\ge3$, then $M_E=W$, and
\[
\frac{M_0}{M_E}=e^{-\kappa\,\frac{\wv^\beta}\beta}\,.
\]
If $2\le\Theta\le3$, then $\big|\wx^{\alpha-2}\,x\cdot v\big|\le\wx^{\alpha-1}\,\wv\lesssim \wv^{\frac1{\rma}(1-\frac1\alpha)+1}=o(\wv^{\beta})$ holds (since $\frac1{\rma}(1-\frac1\alpha)+1 < \beta$, see~\eqref{eq:parameter-consequences}) and
\[
\frac{M_0}{M_E}=\exp\( \kappa\,\chi(\Theta)\,\wx^{\alpha-2}\,x\cdot v-\frac\kappa\beta\,\wv^\beta\)\lesssim e^{-\frac{\kappa}2\frac{\wv^\beta}\beta}\quad\mbox{as}\quad |v|\to+\infty\,.
\]
Since a multiplication by $\wv^{2\,(1-\beta)}$ does not affect this exponential decay,
\be{eq:high-z-domination}
M_0=o\big(M_E\,\wv^{2\,\beta-2}\big)\,.
\ee
After one last enlargement of $R$, the positive term in~\eqref{eq:combined-before-high-z} is absorbed by half of the energy dissipation. This proves
\[
\cL^\star m\le C\,\one_{\{E\le R\}}-c_0\,\kappa\,M_0\,\chi(\Theta)\,\wx^{2\,\alpha-4}\,|x|^2-c_0\,\kappa\,M_E\,\wv^{2\,\beta-2}\,.
\]

If $E\to+\infty$ and $\Theta\le E^{-\rho}$, then $E\asymp\wx^\alpha$ and the mixed term in $F$ is bounded by~\eqref{eq:mixed-small-z}. Hence $\lim_{E\to+\infty}M_0=+\infty$. Otherwise, if $E\to+\infty$ and $\Theta\ge E^{-\rho}$, then
\[
-\ln\zeta_\lambda(\Theta)\le C\,E^\rho=o(E)\,,
\]
so that $M_E=\exp(\kappa\,E-o(E))\to+\infty$ as $E\to+\infty$. We conclude that $m$ is coercive.

\medskip It remains to prove~\eqref{eq:weak-main}. Let us distinguish two cases:

\smallskip\noindent$\bullet$ If $\Theta\le2$, the mixed term in $F$ is bounded and $E\asymp\wx^\alpha$. After increasing $R$,
\[
\ln(e+M_0)\asymp\kappa\,\wx^\alpha\,.
\]
Since $\wx^{2\,\alpha-4}\,|x|^2\asymp\wx^{-2\,(1-\alpha)}$, one obtains
\be{eq:macro-component-log}
\kappa\,M_0\,\wx^{2\,\alpha-4}\,|x|^2\ge c\,\kappa^{1+\frac{2\,(1-\alpha)}{\alpha}}\,\frac{M_0}{\{\ln(e+M_0)\}^{\frac{2\,(1-\alpha)}{\alpha}}}\ge c\,\kappa^{1+\sigma_*}\,\frac{M_0}{\{\ln(e+m)\}^{\sigma_*}}\,.
\ee

\noindent$\bullet$ If  $\Theta\ge E^{-\rho}$, the soft cutoff satisfies $-\ln\zeta_\lambda(\Theta)\le C\,\lambda\,E^\rho=o(\kappa\,E)$. Hence
\[
\ln(e+M_E)\asymp\kappa\,E\,.
\]
Since $\wv^\beta\le\beta\,E$, we have
\[
\wv^{2\,\beta-2}\ge\beta^{2\,\frac{\beta-1}\beta}\,E^{-2\,\frac{1-\beta}\beta}\,.
\]
Consequently,
\be{eq:energy-component-log}
\kappa\,M_E\,\wv^{2\,\beta-2}\ge c\,\kappa^{1+2\,\frac{1-\beta}\beta}\,\frac{M_E}{\{\ln(e+M_E)\}^{2\,\frac{1-\beta}\beta}}\ge c\,\kappa^{1+\sigma_*}\,\frac{M_E}{\{\ln(e+m)\}^{\sigma_*}}\,.
\ee

\noindent\textbf{$\bullet$ Region $\Theta\le E^{-\rho}$.}
By~\eqref{eq:ME-M0-ratio}, $m\asymp M_0$ for large $E$, and~\eqref{eq:macro-component-log} applies because $\chi=1$.

\smallskip\noindent\textbf{$\bullet$ Region $E^{-\rho}\le\Theta\le2$.}
Here $\chi=1$, and the sum of~\eqref{eq:macro-component-log} and~\eqref{eq:energy-component-log} controls $m=M_0+M_E$.

\smallskip\noindent\textbf{$\bullet$ Region $\Theta\ge2$.}
By~\eqref{eq:high-z-domination}, $m\asymp M_E$ for large $E$, and~\eqref{eq:energy-component-log} applies.

\smallskip\noindent These three regions prove~\eqref{eq:weak-main} outside of a compact set in $\RRd$, which completes the proof of Proposition~\ref{prop:main}.

\subsubsection{Extension to the \texorpdfstring{$k^{\rm th}$}{kth} order moment}
The strategy of Proposition~\ref{Prop:Lyapunov} applies except that~\eqref{Eqn:Lyapunov} has to be replaced by~\eqref{eq:weak-main}.
\begin{corollary}\label{Cor:FosterLyapunov} Let $m$ be defined as in Proposition~\ref{prop:main}, $k>0$ and $\ell= 1+\sigma_*$. The function $\widetilde{m}=\ln(e+m)$ satisfies~\eqref{Eqn:Lyapunov}\end{corollary}
\begin{proof} Let $\Phi\in\mathrm C^2(\R_+)$, then using
\[
\nabla_x \big(\Phi(m)\big)= \Phi'(m)\,\nabla_x m\,,\quad\nabla_v \big(\Phi(m)\big)= \Phi'(m)\,\nabla_v m
\]
and
\[
\Delta_v\big(\Phi(m)\big)= \Phi''(m)\,\abs{\nabla_v m}^2 + \Phi'(m)\,\Delta_vm\,,
\]
we have that
\[
\mathcal{L}^*\big(\Phi(m)\big) = \Phi'(m)\, \mathcal{L}^*(m) + \Phi''(m)\,\abs{\nabla_v m}^2\,.
\]
Let us choose $\Phi$ increasing, concave and such that $\Phi(s)= (\ln(e+s))^k$ for $s$ sufficiently large. Then, using $\Phi'(s) = k\,\{\ln(e+s)\}^{k-1}/(e+s)$ for $s$ large and $\Phi''\le 0$, we have
\begin{align*}
\cL^\ast (\widetilde{m}^k)& = \cL^\ast \big(\Phi(m)\big)\le \Phi'(m)\, \mathcal{L}^*(m)\\
&\le \Phi'(m)\( C\,\one_{\{E\le R\}}-c_0\,\kappa^{1+\sigma_*}\,\frac{m}{\{\ln(e+m)\}^{\sigma_*}}\)\\
&\le C\,\one_{\{E\le R\}}-c_0\,\kappa^{1+\sigma_*}\,\frac{k\,m\, \{\ln(e+m)\}^{k-1}}{(e+m)\,\{\ln(e+m)\}^{\sigma_*}}\\
&\lesssim C\,\one_{\{E\le R\}}-c_0\,k\,\kappa^{1+\sigma_*}\, \{\ln(e+m)\}^{k-1-\sigma_*}\lesssim C\,\one_{\{E\le R\}} - \eta\,\widetilde{m}^{k-1-\sigma_*}
\end{align*}
for some $\eta>0$, up to enlargement of $C$ and $R$ if necessary.
\end{proof}

\section{Conclusion: proof of the main results}\label{Sec:Proofs}

Let us collect what we have obtained in Sections~\ref{Sec:Coupling} and~\ref{Sec:Lyapunov} and how intermediate results can be combined to prove our main results. As explained in Section~\ref{Sec:loss}, Lemma~\ref{lem:betterDiss} establishes the exponential decay rate~\eqref{Ineq:Exp} for $\min\{\alpha,\beta\}\ge1$ (wihout any moment assumption) and the result of Proposition~\ref{prop:rates} if \hbox{$0<\min\{\alpha,\beta\}<1$}, under the assumption that $\mathscr K$ is finite (\emph{a priori} boundedness of $\rmL^2$ norms). In~\cite[Theorem~2]{MR4769957}, this property was obtained by the Maximum Principle under a severe restriction on the initial datum. Theorem~\ref{Thm:moments} achieves the same result under the weaker and more natural condition that initial weighted $\rmL^2$ norms are finite.

\begin{proof}[Proof of Theorem~\ref{Thm:moments}] We rely on two main methods (see Figure~\ref{fig:alpha_beta}):
\\[4pt]
(1) \emph{Coupling} of $\rmL^2$-entropy and entropy dissipation estimates: see Section~\ref{Sec:Coupling}. Proposition~\ref{eq:alphabeta23} proves the estimate if $2/3<\min\{\alpha,\beta\}<1$. In the range $\alpha\ge1$ and $\beta\in(0,1)$, we introduce a modified $\rmL^2$-entropy with Lemma~\ref{Lem:Equivalence-eps} (entropy--norm equivalence), Lemma~\ref{lem:Dissipation1} and Lemma~\ref{lem:1} (entropy dissipation) using stronger weights, and complete the moment estimates (Proposition~\ref{prop:energyMomDMS}) using a variant of the Stein-Weiss interpolation theorem (Theorem~\ref{Thm:Stein-Weiss-variant}). The coupling method covers Case (a).
\\[4pt]
(2) Direct $\rmL^2$ moment bounds based on the construction of functions satisfying a \emph{Lyapunov condition}: see Section~\ref{Sec:Lyapunov1}. If $\alpha>0$ and $\beta>1$, this is done in Corollary~\ref{Cor:Stein} and the Lyapunov function of Lemma~\ref{lem:Lapunov}. The case $0<\alpha<1$ and $0<\beta\le1$ is much more delicate, with technical estimates based on elementary computations (Lemma~\ref{lem:propTheta} and Lemma~\ref{lem:zonesTheta}), and the construction of a Lyapunov function which relies on a discussion of various regimes: see Propositions~\ref{prop:LM0} and~\ref{prop:energy} respectively based on Lemma~\ref{lem:macro} and Lemma~\ref{lem:ME}. This function provides us with the weak Foster--Lyapunov estimate of Proposition~\ref{prop:main} which is then used in Corollary~\ref{Cor:FosterLyapunov} to produce a suitable \emph{Lyapunov condition} for Corollary~\ref{Cor:Stein}.
\end{proof}

\begin{proof}[Proof of Corollary~\ref{Cor:rates}] In the \emph{coupling} case (corresponding to (1) in the above proof of Theorem~\ref{Thm:moments}), we obtain the decay estimate on the entropy and the moment bounds simultaneously. In Section~\ref{Sec:Lyapunov}, moment bounds are obtained independently. However, Lemma~\ref{lem:betterDiss} applies in all cases, which provides us with the result.
\end{proof}

\bigskip\noindent{\bf Acknowledgment:} The authors thank the \emph{Conviviality} project (ANR-23-CE40-0003) of the French National Research Agency for support. L.Z.~has received funding from the European Union’s Horizon 2020 research and innovation program under the Marie Skłodowska-Curie grant agreement No 945332.\\
\copyright\,2026\ by the authors. This paper may be reproduced, in its entirety, for non-commercial purposes. \href{https://creativecommons.org/licenses/by/4.0/legalcode}{CC-BY 4.0}

\bibliographystyle{siam}\small
\bibliography{Biblio}
\end{document}